\documentclass{amsart}
\usepackage[utf8]{inputenc}

\usepackage{hyperref}
\usepackage{enumerate} 
\usepackage{upref}
\usepackage{verbatim} 
\usepackage{color}
\usepackage{mathtools,amsmath,amsfonts,amssymb,amsthm, bbm, dsfont}

\usepackage[textsize=tiny]{todonotes}

\newcommand*{\mailto}[1]{\href{mailto:#1}{\nolinkurl{#1}}}

\usepackage{amssymb, amsmath, amsfonts, amsthm}
\usepackage[all]{xy}

\usepackage{marginnote}

\newcommand{\eps}{\epsilon}
\newcommand{\R}{\mathbb{R}}

\newcommand{\N}{\mathbb{N}}

\newcommand{\mF}{\mathcal{F}}
\newcommand{\mD}{\mathcal{D}}

\newcommand{\Cc}{C_0^\infty}

\newcommand{\supp}{\text{\normalfont{supp}}}
\newcommand{\vp}{\varphi}
\newcommand{\hlf}{\frac{1}{2}}
\newcommand{\bbo}{\mathds{1}}
\newcommand{\hY}{\hat{Y}}

\theoremstyle{plain}
\newtheorem{thm}{Theorem}[section]
\newtheorem{lem}[thm]{Lemma}
\newtheorem{prop}[thm]{Proposition}
\newtheorem{defn}[thm]{Definition}

\newtheorem{exmp}[thm]{Example}
\newtheorem{cor}[thm]{Corollary}
\theoremstyle{plain}
\newtheorem{rem}[thm]{Remark}
\newtheorem*{note}{Note}

\usepackage{subfig}
\allowdisplaybreaks

\begin{document}
	
	\title[Lipschitz stability for the HS equation]{Lipschitz stability for the Hunter--Saxton equation}

	\author[K. Grunert]{Katrin Grunert}
	\address{Department of Mathematical Sciences\\ NTNU Norwegian University of Science and Technology\\ NO-7491 Trondheim\\ Norway}
	\email{\mailto{katrin.grunert@ntnu.no}}
	\urladdr{\url{https://www.ntnu.edu/employees/katrin.grunert}}
	
	\author[M. Tandy]{Matthew Tandy}
	\address{Department of Mathematical Sciences\\
	  NTNU Norwegian University of Science and Technology\\
	  NO-7491 Trondheim\\ Norway}
	\email{\mailto{matthew.tandy@ntnu.no}}
	\urladdr{\url{https://www.ntnu.edu/employees/matthew.tandy}}

	\thanks{We acknowledge support by the grants {\it Waves and Nonlinear Phenomena (WaNP)} and {\it Wave Phenomena and Stability - a Shocking Combination (WaPheS)}  from the Research Council of Norway. }  

	\begin{abstract}
		We study Lipschitz stability in time for $\alpha$-dissipative solutions to the Hunter--Saxton equation, where $\alpha \in [0, 1]$ is a constant. We define metrics in both Lagrangian and Eulerian coordinates, and establish Lipschitz stability for those metrics.
	\end{abstract}

	\maketitle
	
\section{Introduction}
	In this paper, we investigate the Lipschitz stability of $\alpha$-dissipative solutions of the initial value problem for the Hunter--Saxton equation,
	\begin{equation}\label{eqn:HS}\tag{HS}
		u_t(x,t) + uu_x(x,t) = \frac{1}{4}\Bigg(\int_{-\infty}^x u_x^2(y,t)\ dy - \int_x^{+\infty}u_x^2(y,t)\ dy\Bigg),
	\end{equation}
	with initial data $u(x,0) = u_0(x)$.
	
		This equation was introduced by Hunter and Saxton as a model for the nonlinear instability in the director field of a nematic liquid crystal \cite{MR1135995}. Further, it is connected to the high frequency limit of the Camassa--Holm equation \cite{MR1677529}.
	
	Solutions to \eqref{eqn:HS} may develop singularities, known as wave breaking, in finite time. That is, $u_x\to-\infty$ spatially pointwise, while $u$ remains continuous and bounded.
	
	One defines the energy density of the solution to be $u_x^2$. Then, at wave breaking, one sees that some of the energy will concentrate on a set of measure zero. Hence, the energy density in general is not absolutely continuous. Instead, the energy is described by a positive Radon measure. The question then becomes, how does one define the solution past wave breaking? This is determined by how one manipulates the energy past wave breaking. In general, one has the freedom to take as much energy away as one pleases \cite{ MR3860266}. Two important cases are well studied. Conservative solutions, whom lose no energy past wave breaking, and dissipative solutions, whom remove the energy that has concentrated on sets of measure zero at wave breaking.  For both the conservative \cite{MR2653980, MR3573580}, and dissipative case \cite{MR2191785}, existence of solutions has been shown. Uniqueness for the dissipative case was shown in \cite{MR2796054}. Further, the dissipative case is the solution with maximal energy loss for a given initial data, as shown in \cite{MR3451933}. The method used in this paper has been applied to the Camassa--Holm equation to prove similar results \cite{MR2372478, MR3005543}, and existence in the case in which only part of the energy may be removed \cite{MR3295963}. A different approach was used to show existence and uniqueness to the differentiated Hunter Saxton equation, $v_t + uv_x = -\frac{1}{2}v^2, v = u_x $ under the assumption that $u(0,t)=0$ for all $t$, on the positive real line, with compactly supported initial data\cite{MR1799274}. Note that solutions of this equation, extended antisymmetric to the whole real line, must not necessarily be solutions to \eqref{eqn:HS}, due to the requirement that $u(0,t) =0$ for all time, which we do not have. 
	
	We are more concerned with the stability of solutions. This builds upon the work of \cite{MR3860266}, for which Lipschitz stability was shown for a given time-dependant distance. We intend to overcome a few assumptions of this paper. Namely, we wish to include the possibility of breaking at time zero, to build a metric that relies on the current energy of the system, rather than the past energy, and to rid the requirement of a purely absolutely continuous initial energy measure in the dissipative case. Lipschitz stability was found for the conservative case using different metrics in \cite{MR3573580, MR3941227}.
	
	Solutions to the problem are found using a generalization of the method of characteristics. For explanatory purposes, formally suppose for now that $u$ is smooth, and its energy density is given by $u_x^2$. Following the work of \cite{MR3573580}, we shift from the Eulerian variable $u$ to Lagrangian variables $ (y,U,V) $, whom satisfy
	\[ y_t(\xi, t) = u(y(\xi, t), t), \]
	\[ U(\xi, t) = u(y(\xi, t), t), \]
	\[ V(\xi, t) = \int_{-\infty}^{y(\xi, t)} u^2_x(z, t)\ dz, \]
	which we can define as long as the energy for the solution $u$ does not concentrate on sets of measure zero, i.e. until wave breaking happens.
	This then gives
	\begin{subequations}\label{eqn:SmoothLagSys}
		\begin{align}
			 y_t(\xi, t) &= U(\xi, t), \\
		 	U_t(\xi, t) &= \frac{1}{2} V(\xi, t) - \frac{1}{4} \lim_{\xi\to\infty}V(\xi, t), \\
			V_t(\xi, t) &= 0.
		\end{align}
	\end{subequations}
	This is a system of ordinary differential equations (ODEs) with initial data
	\begin{subequations}\label{eqn:InitialSmooth}
	\begin{align}
	y(\xi, 0) & = y_0(\xi), \\
	 U(\xi, 0) & = U_0(\xi) = u_0(y_0(\xi)), \\
	 V(\xi, 0) & = V_0(\xi) = \int_{-\infty}^{y_0(\xi)} u^2_x(z, 0)\ dz. 
	 \end{align}
	\end{subequations}
	Assuming energy does not initially concentrate on sets of measure zero, one can take $y_0(\xi) = \xi$. 
	
	Wave breaking then occurs when at least two characteristics meet. The time at which wave breaking occurs is given by
	\begin{equation}\label{def:tauSmooth}
		\tau(\xi) = \begin{cases}
			-2\frac{y_\xi(\xi, 0)}{U_\xi(\xi, 0)}, &\mbox\ U_\xi(\xi, 0)<0,\\
			0, &\mbox\ U_\xi(\xi, 0)=0=y_\xi(\xi, 0),\\
			+\infty, &\mbox\ \text{otherwise.}
		\end{cases}
	\end{equation}
	Up until wave breaking, the solution in Lagrangian coordinates is obtained by solving \eqref{eqn:SmoothLagSys}.  After wave breaking, how one continues is determined by how one manipulates the energy. For conservative solutions, one continues the solution using \eqref{eqn:SmoothLagSys}, retaining the energy in the system. For dissipative solutions, characteristics that interact lose their energy and stick together, given by setting $V_\xi(\xi, t)=0$ for $t\geq \tau(\xi)$. We consider the case of $\alpha$-dissipative solutions, for whom $V_\xi(\xi, t)=(1-\alpha)V_\xi(\xi, 0)$ for $t\geq \tau(\xi) > 0$. In particular, the system \eqref{eqn:SmoothLagSys} is replaced by
	\begin{subequations}\label{eqn:SmoothLagSysD}
		\begin{align}
			 y_t(\xi, t) &= U(\xi, t), \\
		 	U_t(\xi, t) &= \frac{1}{2} V(\xi, t) - \frac{1}{4} \lim_{\xi\to\infty}V(\xi, t),
		\end{align}
	\end{subequations}
	where 
	\begin{equation*}
	V(\xi,t)=\int_{-\infty}^\xi V_\xi(\eta,0)( 1- \alpha\bbo_{\{ r\in\R \mid t\geq \tau(r) > 0\}}(\eta))\ d\eta.
	\end{equation*}
	The more general $\alpha$-dissipative solution \cite{MR3860266} considers the situation in which $\alpha: \R \to [0,1)$, i.e. that the drop in energy depends on the position of the particle.
	
	There is no unique way of defining the initial characteristic $y_0(\xi)$. One cannot assume $y_0(\xi)=\xi$, as this doesn't account for energy initially concentrating on sets of measure zero. Due to this, one defines a transformation from Eulerian to Lagrangian coordinates, as seen in \cite{MR3860266}.
	In Section 2, we introduce the spaces we will be working in, and the mappings used to transform from Eulerian to Lagrangian coordinates and back. In addition, we state some known results we will make use of later in the paper. As the solution at time $t$ depends on how the energy was initially distributed, one must introduce an additional energy variable, $\nu$, which will provide a barrier we must overcome in our construction for the Eulerian metric. Additionally, transforming from Eulerian to Lagrangian variables introduces an extra coordinate, hence multiple Lagrangian coordinates represent the same Eulerian coordinates, thus we introduce equivalence classes, whose elements are related by a relabelling.
	
	Section 3 focuses on the construction of a metric which is Lipschitz in time for the Lagrangian coordinate system. For conservative solutions, the metric can be defined using the normal $L^\infty(\R)$, $L^1(\R)$, and $L^2(\R)$ norms, as no energy in the system has been lost, leading to a smooth metric \cite{MR3573580}. For dissipative solutions, energy may have suddenly dropped in the past, and the challenge is constructing a metric which doesn't jump upwards over these drops in energy, doesn't split apart the multiple Lagrangian solutions representing the same Eulerian solution, and which renders the flow Lipschitz continuous in time, giving the solutions are continuous with respect to the initial data in our metric.
	
	Finally, Section 4 contains our main result. Using the construction in Lagrangian coordinates we can define a metric in Eulerian coordinates. This then inherits the Lipschitz continuity in time from our previous metric. However, the metric must account for all possible drops in energy that could have occurred in the past, that is, all possible past energy densities $\nu$.
	
	\section{The Lagrangian and Eulerian variables}
	Before continuing, we define the sets in which the Eulerian and Lagrangian coordinates lie. We follow the construction in \cite{MR2653980}.
	We begin by defining the Banach space and associated norm
	\[
		E \coloneqq \{ f\in L^\infty(\R) \mid f'\in L^2(\R)\}, \quad\ \|f\|_{E_2} = \|f\|_\infty + \|f'\|_2,
	\]
	and define
	\[
		H_i \coloneqq H^1(\R) \times \R^i, \quad\ i=1,2,
	\]
	with the norms
	\[
		\|(f,a)\|_{H_1} = \sqrt{\|f\|^2_{H^1} + |a|^2}, 
		\quad\
		\|(f,a, b)\|_{H_2} = \sqrt{\|f\|^2_{H^1} + |a|^2 + |b|^2}, 
	\]
	where $H^1(\R)$ is the usual Sobolev space. We then split $\R$ into $(-\infty ,1)$,and $(-1,\infty)$, and choose $\chi^-, \chi^+ \in C^\infty(\R)$ satisfying the following three properties
	\begin{itemize}
		\item $ \chi^- + \chi^+ = 1 $,
		\item $ 0\leq \chi^+ \leq 1 $,
		\item $ \supp(\chi^-) \subset (-\infty, 1) $ and $ \supp(\chi^+) \subset (-1, \infty) $.
	\end{itemize}
	We now introduce the mappings
	\begin{subequations}
	\begin{align}\label{map:R}
		R_1: H_1 \to E&  \quad\ (f, a) \mapsto f + a \cdot \chi^+,\\
		R_2: H_2 \to E&  \quad\ (f, a, b) \mapsto f + a \cdot \chi^+ + b \cdot \chi^-.
	\end{align}
	\end{subequations}
	These mappings are linear and continuous, due to functions in $H^1(\R)$ being continuous. They are also injective. We show this for $R_2$, and $R_1$ follows with $b=0$. If we have two equal elements $F$ and $G$ in the codomain, then there exists $f,g \in H^1(\R)$, and $a_f,b_f,a_g,b_g\in\R$ such that
	\[
		 f(\xi) + a_f \cdot \chi^+(\xi) + b_f \cdot \chi^-(\xi) = F(\xi) = G(\xi) = g(\xi) + a_g \cdot \chi^+(\xi) + b_g \cdot \chi^-(\xi).
	\]
	for all $\xi \in \R$.
	Taking the limits at $\pm\infty$, we find $a_f = a_g$ and $b_f = b_g$. It then immediately follows that $f=g$ as required.
	
	From these we define the following Banach spaces and associated norms,
	\[
		E_1 \coloneqq R_1(H_1), \quad\ \|f\|_{E_1} = \|R_1^{-1}(f)\|_{H_1},
	\]
	\[
		E_2 \coloneqq R_2(H_2), \quad\ \|f\|_{E_2} = \|R_2^{-1}(f)\|_{H_2}.
	\]
	\begin{rem}[The choice of $\chi$ does not change $E_1$]
		Consider $\chi^+$ and $\hat{\chi}^+$ satisfying the above conditions. Define $R_1$ and $\hat{R}_1$ as one would expect, reflecting \eqref{map:R}. We show $R_1(H_1) = \hat{R}_1(H_1)$. Consider $f\in R_1(H_1)$. Then there exists $g \in H^1(\R)$ and $a\in \R$ such that
		\[
			f = g + a \cdot \chi^+.
		\]
		Noting that $\chi^+ - \hat{\chi}^+$ is in $C^\infty_c(\R)$, we have
		\[
			f - a \cdot  \hat{\chi}^+ = g + a \cdot (\chi^+ - \hat{\chi}^+) \in H^1(\R),
		\]
		therefore $f = f -a \cdot \hat{\chi}^+ + a \cdot \hat{\chi}^+$ is in $\hat{R}_1(H_1)$, thus demonstrating $R_1(H_1) \subset \hat{R}_1(H_1)$. The same approach can be used to show $\hat{R}_1(H_1) \subset R_1(H_1)$.
		
		It can also be shown that $E_2$ does not rely on the choice of $\chi^-$ and $\chi^+$.
	\end{rem}
	Using these, we define the Banach space $B$, and associate with it the expected norm
	\[
		B \coloneqq E_2 \times E_2 \times E_1 \times E_1, \quad\, \|(f_1,f_2,f_3,f_4)\|_B = \|f_1\|_{E_2} + \|f_2\|_{E_2} + \|f_3\|_{E_1} + \|f_4\|_{E_1}.
	\]
	Wave breaking may occur at time zero, or may have even occurred in the past. The measure $\mu$ corresponds to the energy of the system at time zero. To model previous wave breaking and the corresponding energy loss, an additional energy measure $\nu$ must be supplied. This variable carries the initial energy forward in time (i.e. $\nu_t(\R, t)=0$, as we will see when mapping from Lagrangian to Eulerian coordinates). Corresponding to $\nu$ when transforming to Lagrangian coordinates, a variable $H$ is introduced. This will also preserve the energy forward in time. The variable $V$ corresponds to the current energy $\mu$. This variable is necessary for the construction of a semigroup of solutions in Lagrangian coordinates.
	
		We begin with the set of Eulerian coordinates:
	
	\begin{defn}[Set of Eulerian coordinates - $\mD$]\label{def:D}
		The set $\mD$ contains all Eulerian variables $ Y=(u,\mu,\nu) $ satisfying the following
		\begin{itemize}
			\item $ u\in E_2 $,
			\item $ \mu \leq \nu \in \mathcal{M}^+(\R) $,
			\item $ \mu\big((-\infty, x)\big) - \chi_+(x)\mu(\R) \in L^2(\R) $,
			\item $ \mu_{ac} = u_x^2\ dx$, 
			\item If $\alpha=1, \nu_{ac} = \mu =  u_x^2\ dx, $
			\item If $0\leq \alpha < 1$, $ \frac{d\mu}{d\nu}(x) \in \{1, 1-\alpha\}$, and $\frac{d\mu}{d\nu} = 1$ if $u_x(x)<0$,
		\end{itemize}
		where $\mathcal{M}^+(\R)$ is the set of all finite, positive Radon measures on $\R$.
	\end{defn}
	
	Followed by the Lagrangian coordinates:
	
	\begin{defn}[Set of  Lagrangian coordinates - $\mF$]
		Let the set $\mF$ be the set of all $ X=(y,U,H,V)$, where $(y-\text{id},U,H,V)\in B$, satisfying the following properties
		\begin{itemize}
			\item $y-id, U, H, V \in W^{1, \infty}(\R)$,
			\item $y_\xi, H_\xi \geq 0$, and there exists a constant $c$ such that $0<c<y_\xi + H_\xi$ a.e.,
			\item $y_\xi V_\xi = U_\xi^2$,
			\item $ 0\leq V_\xi \leq H_\xi $ a.e.,
			\item If $\alpha = 1, y_\xi(\xi) = 0$ implies $V_\xi(\xi) = 0$, $y_\xi(\xi)>0$ implies $V_\xi(\xi) = H_\xi(\xi)$ a.e.,
			\item If $0\leq \alpha < 1$, there exists $\kappa: \R \to \{(1-\alpha), 1\}$ such that $V_\xi(\xi) = \kappa(\xi) H_\xi(\xi)$ a.e., with $\kappa(\xi) = 1$ for $U_\xi(\xi) < 0$.
		\end{itemize}
		Define the set $\mF_0$ as
		\[
			\mF_0 \coloneqq \big\{X\in\mF \mid y+H=\text{id} \big\}.
		\]
	\end{defn}

	The $\alpha$-dissipative solution $X(t)$ for the equation \eqref{eqn:HS} in Lagrangian variables is then given by the following ODE system, with initial data $X(0)\in\mF$,
	\begin{subequations}\label{eqn:LagSys}
		\begin{align}
			 y_t(\xi, t) &= U(\xi, t), \\
			 U_t(\xi, t) &= \frac{1}{2} V(\xi, t) - \frac{1}{4} \lim_{\xi\to\infty}V(\xi, t), \\ \label{eqn:LagSys3}
			 H_t(\xi, t) &= 0,\\ \label{eqn:LagSys4}
			 V(\xi, t) &= \int_{-\infty}^\xi V_\xi(\eta, 0 ) (1 - \alpha\mathds{1}_{\{r \in \R\mid t\geq \tau(r) > 0\}})(\eta)\ d\eta,
		\end{align}
	\end{subequations}
	for whom existence and uniqueness was shown in \cite{MR3860266}, in addition to the fact that the wave breaking time is given by
	\begin{equation}\label{def:tau}
		\tau(\xi) = \begin{cases}
			-2\frac{y_\xi(\xi, 0)}{U_\xi(\xi, 0)}, &\mbox\ U_\xi(\xi, 0)<0,\\
			0, &\mbox\ U_\xi(\xi, 0)=0=y_\xi(\xi, 0),\\
			+\infty, &\mbox\ \text{otherwise.}
		\end{cases}
	\end{equation}
	
	Transforming from Eulerian to Lagrangian coordinates and back is achieved by the following mappings, which are inverses, with respect to equivalence classes, of each other \cite{MR3573580, MR3860266} and which developed from the transformations defined for the Camassa--Holm equation in \cite{MR2372478}.
	\begin{defn}[Mapping $L:\mD\to\mF_0$]\label{map:EultoLag}
		The following defines the mapping $L:\mD\to\mF_0$, from Eulerian to Lagrangian coordinates,
		\begin{subequations}
			\begin{align}
				y (\xi) &= \sup\{ x\in\R \mid x + \nu\big((-\infty ,x)\big) < \xi \},\\
				U(\xi) &= u(y(\xi)),\\
				H(\xi) &= \xi - y(\xi),\\
				V(\xi) &= \int_{-\infty}^\xi H_\xi(\eta) \frac{d\mu}{d\nu}\circ(y(\eta))\ d\eta. \label{map:L}
			\end{align}
		\end{subequations}
	\end{defn}

	\begin{defn}[Mapping $M:\mF \to \mD$]\label{map:LagtoEul}
		The following defines the mapping $M:\mF\to\mD$, from Lagrangian to Eulerian coordinates,
		\begin{subequations}
			\begin{align}
			u(x) &= U(\xi), \quad\ \text{ for all } \xi\in\R \text{ such that } x=y(\xi),\\
			\mu &= y_{\#}(V_\xi\ d\xi),\\
			\nu &= y_{\#}(H_\xi\ d\xi).
			\end{align}
		\end{subequations}
		Here, we have used the push forward measure for a measurable function $f$ and $\mu$-measurable set $f^{-1}(A)$, i.e.,
		\[
		f_{\#}(\mu)(A) \coloneqq \mu(f^{-1}(A)).
		\]
	\end{defn}
	The mapping $L$ introduces an additional coordinate when mapping from Eulerian to Lagrangian coordinates, hence the mapping is not one-to-one. On the other hand, one can introduce an equivalence relation on $\mF$, equating Lagrangian coordinates representing the same Eulerian coordinates.
	\begin{defn}[Equivalence relation on $\mF$]\label{defn:equivRel}
		Let $G$ be the group of homeomorphisms $f:\R\to\R$ satisfying
		\begin{equation}\label{proprt:eqf}
			f-\text{id} \in W^{1,\infty}(\R), \quad\ f^{-1}-\text{id} \in W^{1,\infty}(\R), \quad\  f_\xi - 1 \in L^2(\R).
		\end{equation}
		We define the group action $\bullet:\mF\times G \to \mF$, called the relabelling of $X$ by $f$, as
		\[
			(X,f) \mapsto X\bullet f=(y\circ f, U\circ f, H\circ f, V\circ f).
		\]
		Hence, one defines the equivalence relation $\sim$ on $\mF$ by
		\[
			X_A\sim X_B \text{ if there exists }f\in G \text{ such that } X_A = X_B\bullet f.
		\]
		Finally, define the mapping $\Pi:\mF \to \mF_0$, which gives one representative in $\mF_0$ for each equivalence class,
		\[
			\Pi(X) = X\bullet (y + H)^{-1}.
		\]
	\end{defn}
	\begin{note}
		We have used in our definition for $\Pi$ that $(y + H)^{-1} \in G$.
		We will simply write $\Pi X$, though this is not a linear transformation.
	\end{note}
	\begin{lem}\cite[Proposition 3.5]{MR3860266}
		Let $X, \tilde{X} \in \mF$, and assume $X\sim \tilde{X}$, then
		\[
			M(X) = M(\tilde{X}).
		\] 
	\end{lem}
	\begin{proof} Let $f\in G$ be such that $\tilde{X}=X\bullet f$. As $f$ is a bijection, 
		\begin{alignat*}{2}
			\tilde{u}(x) 
			&= \tilde{U}(\xi), \quad\ &&\phantom{i}\text{for all } \xi\in\R \text{ such that } x=\tilde{y}(\xi),\\
			&= (U\circ f)(\xi), \quad\  &&\text{ for all } \xi\in\R \text{ such that } x=(y\circ f)(\xi),\\
			&= U(\eta), \quad\ &&\text{ for all }\eta=f(\xi)\in\R \text{ such that } x=y(\eta),\\
			&= u(x).
		\end{alignat*}
		For any Borel set $A\subset \R$, we have, using the substitution $\eta = f(\xi)$,
		\begin{align*}
		\tilde{\mu}(A) 
		&= \int_{(y\circ f)^{-1}(A)} (V\circ f)_\xi(\xi)\ d\xi\\
		&= \int_{y^{-1}(A)} V_\xi(\eta)\ d\eta 
		= \mu(A).
		\end{align*}
		The proof for $\nu$ follows from the same calculations as $\mu$.
	\end{proof}
	Relabelling can be done either initially, or after a given time, and one obtains the same solution, as the following proposition states.
	\begin{prop}\label{prop:solRel}\cite[Proposition 3.7]{MR3860266}
		Define the solution operator $S_t:\mF \to \mF$, $X \mapsto S_t(X)$ as giving the solution at time $t$ to the ODE system \eqref{eqn:LagSys} with initial data $X\in\mF$. Then
		\[
			S_t(X \bullet f) = S_t(X) \bullet f,
		\]
		for any $f\in G$.
	\end{prop}
	For completeness, we include the definition of a weak $\alpha$-dissipative solution to \eqref{eqn:HS}. Existence of solutions, using the generalized method of characteristics, was found in \cite{MR3860266}.
	\begin{defn}
		$(u,\mu,\nu)$ is a weak $\alpha$-dissipative solution to \eqref{eqn:HS} with initial data $(u_0, \mu_0, \nu_0)\in\mD$, if
		$(u, \mu, \nu) \in \mD$ satisfies the initial data, and
		\begin{subequations}
			\begin{align}
				 u &\in C^{0,\hlf}(\R\times[0,T],\R), \quad\ \text{ for all } T\in [0,+\infty),\\
				 \nu &\in C_{weak*}([0,+\infty], \mathcal{M}^+(\R)),\\
				 \nu(t)(\R) &= \nu_0(\R), \quad\ \text{ for all } t\in [0,+\infty),\\
				 d\mu(t) &= d\mu_{ac}(t)^- + (1-\alpha)d\mu_s(t)^-,\\
				 \mu(s) &\overset{\ast}{\rightharpoonup} \mu(t), \quad\  \text{ for all } t \in [0,+\infty) \text{ from above},\\
				 \mu(s) &\overset{\ast}{\rightharpoonup} \mu(t)^{-}, \quad\ \text{ for all } t \in [0,+\infty) \text{ from below},
			\end{align}
		\end{subequations}
		and, for all test functions $\vp \in \Cc(\R\times [0,+\infty))$, \eqref{eqn:HS} is satisfied in the distributional sense, that is
		\begin{equation}
			\int_0^{\infty}\int_\R \bigg(u\vp_t + \hlf u^2 \vp_x - \frac{1}{4}\bigg(\int_{-\infty}^x u_x^2\ dy - \int_x^{\infty}u_x^2\ dy\bigg)\vp\bigg)\ dx dt = -\int_\R u_0 \vp_0\ dx,
		\end{equation}
		where $\vp_0(x) = \vp(x, 0)$. Further, for each non-negative test function $\phi \in \Cc(\R\times [0,+\infty))$, one must have
		\begin{equation}
			\int_0^{+\infty} \int_\R (\phi_t + u \phi_x)\ d\mu(t) dt \geq - \int_\R \phi_0\ d\mu_0.
		\end{equation}
	\end{defn}
	
	For a complete work through of an $\alpha$-dissipative problem, see Example \ref{exmp:adiss}.
		
	\section{Lipschitz stability in Lagrangian coordinates}
	
	We now have the necessary prerequisites to start constructing a metric in Lagrangian coordinates such that the solution to the ODE system \eqref{eqn:LagSys} is Lipschitz continuous.
	
	Before constructing our metric, we ease the notation. Given $X_i$, $X_j \in \mF$, we define the following sets
	\begin{subequations} \label{eqn:Rsplit}
		\begin{align}
			A_i(t) &\coloneqq \big\{ \xi \in \R \mid U_{i, \xi}(\xi, t) \geq 0 \big\},\\
			A_{i,j}(t) &\coloneqq A_i(t) \cap A_j(t),\\
			B_{i, j}(t) &\coloneqq \big\{ \xi \in \R \mid t <  \tau_{i}(\xi) = \tau_{j}(\xi) <+\infty \big\},\\
			\Omega_{i,j}(t) &\coloneqq A_{i,j}(t) \cup 	B_{i,j}(t).
		\end{align}
	\end{subequations}
	We use these to split the real line into two halves. Define, for $X_1$, $X_2 \in \mF$,
	\begin{equation}
		G_{12}(\xi, t) \coloneqq |V_{1,\xi} - V_{2,\xi}|(\xi, t) \mathds{1}_{\Omega_{12}(t)}(\xi) + \big( V_{1,\xi} \vee V_{2,\xi} \big)(\xi, t) \mathds{1}_{\Omega_{12}^c(t)}(\xi),
	\end{equation}
	where we have used the notation $a\vee b = \max\{a, b\}$.
	
	We can now define our metric $ d:\mF^2\to\R $ as
	\begin{align}\nonumber
		d(X_1,X_2) 
		&\coloneqq 	\|y_1-y_2\|_\infty 
		+ \|U_1-U_2\|_\infty 
		+ \|y_{1,\xi}-y_{2,\xi}\|_2\\ \label{def:d}
		&\quad\ + \|U_{1,\xi}-U_{2,\xi}\|_2 
		+ \|H_1 - H_2\|_\infty 
		+ \|G_{12} \|_1 + \| G_{12} \|_2.
	\end{align}

	A naive approach would be to use the $L^1(\R)$ and $L^2(\R)$ norms of $V_{1,\xi}-V_{2,\xi}$. However upon wave breaking, these norms could suddenly jump upwards. Consider, for instance, the fully dissipative case, i.e. $\alpha=1$, with $X_1$ and $X_2$ in $\mF$ such that $V_{1,\xi}=V_{2,\xi}$ initially. Suppose the first does not break, while the second does. The norm $\|V_{1,\xi}-V_{2,\xi}\|_1$ would initially be zero and would jump upwards and hence become strictly positive after wave breaking. We avoid this by using the norms of $G_{12}$ instead. These are designed to drop after wave breaking in every situation, and thus they are shrinking as time moves forward.
		
	To ensure that $d$ is indeed a metric, we must confirm that the triangle inequality is satisfied for the $G_{12}$ terms.
	\begin{prop}
		The function $d: \mF^2\to\R$ given by \eqref{def:d} satisfies the triangle inequality.
	\end{prop}
	\begin{proof}
		The triangle inequality is immediate for all the norms in $d$ with the exception of the $L^1(\R)$ and $L^2(\R)$ norms of $G_{12}$. To ensure these satisfy the triangle inequality, we show that, for  all $X_1$, $X_2$, $X_3 \in \mF$, we have
		\[
			G_{13}(\xi, t) \leq G_{12}(\xi, t) + G_{23}(\xi, t).
		\]
		We introduce the following notation
		\[
			g_{12}(\xi, t) = |V_{1,\xi} - V_{2,\xi}|(\xi, t) \mathds{1}_{\Omega_{12}(t)}(\xi),
		\]
		\[
			\tilde{g}_{12}(\xi, t) = \big( V_{1,\xi} \vee V_{2,\xi} \big)(\xi, t) \mathds{1}_{\Omega_{12}^c(t)}(\xi),
		\]
		which yields
		\[
			G_{12}(\xi, t) = g_{12}(\xi, t) + \tilde{g}_{12}(\xi, t).
		\]
		We begin by noting the following:
		\begin{itemize}
		\item If $ \xi\in\Omega_{13}(t)$, then $ \xi \in \Omega_{12}(t) \cap \Omega_{23}(t) $ 
			or $ \xi \in \Omega_{12}^c(t) \cap \Omega_{23}^c(t) $, but not both.
		\item If $ \xi \in \Omega_{13}^c(t), $ then $\xi \in  \Omega_{12}^c(t) \cap \Omega_{23}^c(t)$, unless one of the following two cases occurs: 
		\begin{itemize}
		\item If $\xi \in A_{12}(t)$ and $ \xi \notin A_3(t)$, or $\xi \in B_{12}(t)$, then $\xi \in \Omega_{12}(t) \cap  \Omega_{23}^c(t)$.
		\item If $\xi \in A_{23}(t)$ and  $\xi \notin A_1(t)$, or $ \xi \in B_{23}(t)$, then  $ \xi \in \Omega_{12}^c(t) \cap  \Omega_{23}(t)$.
		\end{itemize}
		\end{itemize}
		Note the sets $\xi$ ends up in are all disjoint.
		
		Further, for $a,b,c \geq 0$, we have the following inequalities,
		\begin{subequations}
		\begin{align}\label{ineq:bas1}	
			|a- b| &\leq a \vee b,\\
			a \vee b & \leq a \vee c + |b-c|.
		\end{align}
		\end{subequations}
		We hence strategically use the required inequality for each of the cases above:
		\begin{itemize}
			\item If $ \xi\in\Omega_{13}(t)$,  then either  $ \xi \in \Omega_{12}(t) \cap \Omega_{23}(t) $, and 
			\begin{equation*}
			 	|V_{1,\xi} - V_{3, \xi}|(\xi ,t) \leq |V_{1,\xi} - V_{2, \xi}|(\xi ,t) + |V_{2,\xi} - V_{3, \xi}|(\xi ,t)
			\end{equation*}
			or $ \xi \in \Omega_{12}^c(t) \cap \Omega_{23}^c(t) $ and 
			\begin{equation*}
				|V_{1,\xi} - V_{3, \xi}|(\xi ,t) \leq (V_{1,\xi} \vee V_{2, \xi})(\xi ,t) + (V_{2,\xi} \vee V_{3, \xi})(\xi ,t), 
			\end{equation*}
			giving
			\begin{align*}
				g_{13}(\xi ,t) 
				&\leq \Big( 
				g_{12}(\xi ,t)\bbo_{\Omega_{23}(t)}(\xi ,t)
				+ g_{23}(\xi ,t)\bbo_{\Omega_{12}(t)}(\xi ,t)\\
				&\quad\ + \tilde{g}_{12}(\xi ,t)\bbo_{\Omega^c_{23}(t)}(\xi ,t) 
				+ \tilde{g}_{23}(\xi ,t) \bbo_{\Omega^c_{12}(t)}(\xi ,t)
				\Big)\bbo_{\Omega_{13}(t)}(\xi ,t).
			\end{align*}
			\item If $ \xi\in\Omega_{13}^c(t)$,  we either have  $\xi \in  \Omega_{12}^c(t) \cap \Omega_{23}^c(t)$ and 
			\begin{equation*}
				(V_{1,\xi} \vee V_{3,\xi})(\xi ,t)
				\leq (V_{1,\xi} \vee V_{2,\xi})(\xi ,t) + (V_{2,\xi} \vee V_{3,\xi})(\xi ,t),
			\end{equation*}
			$\xi \in \Omega_{12}(t) \cap  \Omega_{23}^c(t)$ and 
			\begin{equation*}
				(V_{1,\xi} \vee V_{3,\xi})(\xi ,t) 
				\leq |V_{1,\xi}-V_{2,\xi}|(\xi ,t)+(V_{2,\xi} \vee V_{3,\xi})(\xi ,t) , 
			\end{equation*} 
			or $ \xi \in \Omega_{12}^c(t) \cap  \Omega_{23}(t) $ and 
			\begin{equation*}
				(V_{1,\xi} \vee V_{3,\xi})(\xi ,t) 
				\leq (V_{1,\xi} \vee V_{2,\xi})(\xi ,t)+ |V_{2,\xi}-V_{3,\xi}|(\xi ,t),
			\end{equation*}
			giving 
			\begin{align*}
				\tilde{g}_{13}(\xi ,t) 
				&\leq \Big(
				\tilde{g}_{12}(\xi ,t)(\bbo_{\Omega_{23}(t)}(\xi ,t) + \bbo_{\Omega^c_{23}(t)}(\xi ,t))
				+ g_{12}(\xi ,t)\bbo_{\Omega^c_{23}(t)}(\xi ,t)\\
				&\quad\ +\tilde{g}_{23}(\xi ,t)(\bbo_{\Omega_{12}(t)}(\xi ,t) + \bbo_{\Omega^c_{12}(t)}(\xi ,t))
				+ g_{23}(\xi ,t)\bbo_{\Omega^c_{12}(t)}(\xi ,t) 
				\Big)\bbo_{\Omega^c_{13}(t)}(\xi ,t).
			\end{align*}
		\end{itemize}
		As each part of these sums lie on disjoint sets, we indeed have
		\[
			G_{13}(\xi, t) \leq G_{12}(\xi, t) + G_{23}(\xi, t), \quad\ \text{for all } (\xi, t) \in \R \times [0,+\infty).
		\]
		As all the involved functions are positive, one can apply both the $L^1(\R)$ and the $L^2(\R)$ norm on either side of the above inequality, and use the triangle inequality, to obtain the required result.
	\end{proof}
	We are now ready to establish stability.
	\begin{thm}\label{thm:est:d}
		Let $X_1(t)$ and $X_2(t)$ be the solutions of the system \eqref{eqn:LagSys} with initial data $X_{1}(0)$ and $X_{2}(0)$ in $\mF$, respectively. Then
		\[
			d(X_1(t),X_2(t)) \leq e^t d(X_{1}(0),X_{2}(0)).
		\]
	\end{thm}
	\begin{proof}
		We derive inequalities for each of the terms in our metric.
		To do this, we focus first on the metric $D:\mF^2\to\R$, given by
		\begin{align}
			D(X_1, X_2) \label{def:D}
			&\coloneqq d(X_1, X_2) - \|G_{12}\|_1 - \|G_{12}\|_2\\ \nonumber
			&= \|y_1-y_2\|_\infty 
			+ \|U_1-U_2\|_\infty 
			+ \|y_{1,\xi}-y_{2,\xi}\|_2\\ \nonumber
			&\quad\ + \|U_{1,\xi}-U_{2,\xi}\|_2 
			+ \|H_1 - H_2\|_\infty.
		\end{align}
		We do not need an estimate for the norm involving $H$, as it is constant in time.
		Beginning with the $y$ terms, we have from \eqref{eqn:LagSys}
		\[
			|(y_{1} - y_{2})(\xi, t)| \leq |(y_{1} - y_{2})(\xi, 0)| + \int_0^t |(U_{1}-U_{2})(\xi, s)|\ ds,
		\]
		and hence
		\begin{equation}\label{ineq:thm:yEstInf}
			\|(y_1 - y_2)(\cdot, t)\|_\infty \leq \|(y_1 - y_2)(\cdot,0)\|_\infty + \int_0^t \|(U_1 - U_2)(\cdot, s)\|_\infty\ ds.
		\end{equation}
		We also have,
		\begin{equation}\label{ineq:thm:yEst2}
			\|(y_{1,\xi} - y_{2,\xi})(\cdot, t)\|_{2} \leq \|(y_{1, \xi} - y_{2,\xi})(\cdot, 0)\|_{2} + \int_0^t \|(U_{1,\xi}-U_{2,\xi})(\cdot, s)\|_{2}\ ds,
		\end{equation}
		which follows immediately from the Lagrangian ODE system \eqref{eqn:LagSys}, and Minkowski's integral inequality.
		
		Set $V_\infty(t) \coloneqq \displaystyle \lim_{\xi\to +\infty} V(\xi, t)$. Then we have for the $U$ terms,
		\begin{equation}\label{ineq:thm:UEst}
			(U_1 - U_2)(\xi, t) = (U_1 - U_2)(\xi, 0) + \int_0^t (U_{1,t} - U_{2,t})(\xi, s)\ ds,
		\end{equation}
		and for the integral on the RHS,
		\begin{align*}
			\int_0^t (U_{1,t} - U_{2,t})(\xi,s)\ ds
			&= \int_0^t \frac{1}{2}(V_1 - V_2)(\xi,s) - \frac{1}{4} (V_{1,\infty} - V_{2,\infty})(s)\ ds\\
			&= \int_0^t \frac{1}{4}(V_1 - V_2)(\xi,s) \\
			&\hphantom{= \int_0^t } + \frac{1}{4}(V_1 - V_2)(\xi,s) - \frac{1}{4} (V_{1,\infty} - V_{2,\infty})(s) \ ds\\
			&= \frac{1}{4} \int_0^t \Big[\int_{-\infty}^\xi (V_{1,\xi} - V_{2,\xi})(\eta,s)\ d\eta\\
			&\hphantom{= \frac{1}{4} \int_0^t \Big[} - \int_\xi^{+\infty} (V_{1,\xi} - V_{2,\xi})(\eta,s) \ d\eta\Big]\ ds.
		\end{align*}
		Substituting into \eqref{ineq:thm:UEst} and taking the absolute value, we have
		\begin{equation*}
			|U_1 - U_2|(\xi, t) \leq |U_1 - U_2|(\xi, 0) + \frac{1}{4} \int_0^t \int_\R |V_{1,\xi} - V_{2,\xi}|(\eta, s)\ d\eta\ ds.
		\end{equation*}
		Concentrating on the integral on the RHS, we obtain
		\begin{align*}
			\int_\R |V_{1,\xi}-V_{2,\xi}|(\eta, s)\ d\eta 
			&\leq \int_{\Omega_{12}(s)} |V_{1,\xi} - V_{2,\xi}|(\eta, s)\ d\eta\ + \int_{\Omega_{12}^c(s)} ( V_{1,\xi} \vee V_{2,\xi} ) (\eta, s)\ d\eta\\
			&= \int_\R G_{12}(\eta ,s)\ ds.
		\end{align*}
		Thus, after taking the $L^\infty(\R)$ norm, we end up with
		\begin{equation}\label{ineq:thm:UEstInf}
			\|(U_1 - U_2)(\cdot, t)\|_\infty 
			\leq \|(U_1 - U_2)(\cdot, 0)\|_\infty
			+\frac{1}{4} \int_0^t \|G_{12}(\cdot, s)\|_1\ ds.
		\end{equation}
		For the $L^2(\R)$ norm involving the $U_\xi$'s, we use Minkowski's integral inequality, giving
		\begin{align*}
			\|(U_{1,\xi} - U_{2,\xi})(\cdot,t)\|_{2}
			\leq\|(U_{1,\xi} - U_{2,\xi})(\cdot,0)\|_{2} + \hlf \int_0^t \|(V_{1,\xi} - V_{2,\xi})(\cdot,s)\|_2\ ds.
		\end{align*}
		Using that we integrate on two disjoint sets and \eqref{ineq:bas1}, we have
		\begin{align*}
			\Big(\int_\R |V_{1,\xi} - V_{2,\xi}|^2(\xi, s)\ d\xi \Big)^{\frac{1}{2}} 
			&\leq  \bigg(\int_{\Omega_{12}(s)} |V_{1,\xi} - V_{2,\xi}|^2(\xi, s)\ d\xi\\ 
			&\hphantom{\leq  \Big(\int_{\Omega_{12}(s)}} 
			+ \int_{\Omega_{12}^c(s)} \big( V_{1,\xi} \vee V_{2,\xi} \big)^2(\xi, s)\ d\xi \bigg)^{\frac{1}{2}}\\
			&= \left(\int_\R |G_{12}(\xi ,s)|^2\ d\xi \right)^\hlf,
		\end{align*}
		and hence
		\begin{equation}\label{ineq:thm:UEst2}
			\|U_{1,\xi}(\cdot, t) - U_{2,\xi}(\cdot,t)\|_{2}
			\leq \|U_{1,\xi}(\cdot,0) - U_{2,\xi}(\cdot,0)\|_{2} 
			+ \hlf\int_0^t \|G_{12}(\cdot, s)\|_2\ ds. 
		\end{equation}
		Combining \eqref{ineq:thm:yEstInf}, \eqref{ineq:thm:yEst2}, \eqref{ineq:thm:UEstInf}, and  \eqref{ineq:thm:UEst2} together, yields
		\begin{equation}\label{ineq:D}
		\begin{split}
			D(X_1(t), X_2(t))
			&\leq D(X_1(0), X_2(0))\\
			&\quad + \int_0^t \bigg(D(X_1(s), X_2(s))+ \frac{1}{4} \| G_{12}(\cdot, s) \|_1
			+ \hlf\| G_{12}(\cdot, s) \|_2 \bigg) ds.
		\end{split}
		\end{equation}
		Thus, it remains to show that $G_{12}(\xi,t)$ is a decreasing function with respect to time.
		
		As, for all $\xi \in \R$, the $V_{\xi}(\xi,t)$ are decreasing functions in time, $(V_{1,\xi} \vee V_{2,\xi})(\xi,t)$ is a decreasing function in time. Should no wave breaking occur, then the difference $\vert V_{1,\xi}-V_{2,\xi}\vert (\xi,t)$ will remain unchanged. Should both break at the same time, then the difference will decrease, as after wave breaking
		\[
			|V_{1,\xi} - V_{2,\xi}|(\xi, t) = 	(1-\alpha)|V_{1,\xi} - V_{2,\xi}|(\xi, 0).
		\]
		Finally, one has to deal with the case of being in $\Omega^c_{12}(0)$ initially, then ending in $\Omega_{12}(t)$, as can happen if one has broken (or will never break) and the other one will break in the future. Define $a\wedge b \coloneqq \min\{a, b\}$. After breaking, one can write the difference as
		\[
			|V_{1,\xi} - V_{2,\xi}|(\xi, t) = \big( V_{1,\xi} \vee V_{2,\xi} \big)(\xi, t) - \big( V_{1,\xi} \wedge V_{2,\xi} \big)(\xi, t) \leq \big( V_{1,\xi} \vee V_{2,\xi} \big)(\xi, 0)
		\]
		due to the fact that, as mentioned previously, the maximum is a decreasing function of time, and the $V_\xi$'s are both positive. Thus one can conclude
		\begin{equation}\label{ineq:thm:GEst12}
			\|G_{12}(\cdot, t)\|_{1} \leq \|G_{12}(\cdot, 0)\|_{1} \quad \text{ and }\quad  \|G_{12}(\cdot, t)\|_{2} \leq \|G_{12}(\cdot, 0)\|_{2}.
		\end{equation}
		Combining this with inequality \eqref{ineq:D} and recalling \eqref{def:D}, one has
		\[
			d(X_1(t), X_2(t)) \leq d(X_1(0), X_2(0)) + \int_0^t d(X_1(s), X_2(s))\ ds
		\]
		and Gr\"{o}nwall's inequality gives the required result.
	\end{proof}
	
	This metric faces a major problem: Although two different members of an equivalence class in Lagrangian coordinates represent the same element in Eulerian coordinates, they may have a distance greater than zero. This is demonstrated in the following example.
	
	\begin{exmp}\label{exmp:distanceWrongEquivClasses}
		Consider the HS equation with initial data,
		\[
		u_0(x) = \begin{cases}
		1, &\mbox\ x\leq 0,\\
		1-x, &\mbox\ 0<x\leq1,\\
		0, &\mbox\ 1<x,
		\end{cases}
		\quad\
		\nu_0 = \mu_0 = u_{0,x}^2(x)\ dx.
		\]
		As our initial characteristic we can use $y_0(\xi) = \xi$, since neither energy concentrates on sets of measure zero nor $u_{0,x}(x)$ is unbounded. Furthermore, $U_0(\xi) =u_0(y_0(\xi))= u_0(\xi)$ by \eqref{eqn:InitialSmooth}. We then find, using \eqref{def:tau}, that wave breaking will only occur for $\xi\in (0,1)$ and, in particular, $\tau(\xi)=2$ for all $\xi \in (0,1)$. For $t<2$, i.e. before wave breaking occurs, the solution is given by \eqref{eqn:LagSys} and reads
		\[
			V(\xi,t) = 
			\begin{cases}
				0, &\mbox\ \xi\leq 0,\\
				\xi, &\mbox\ 0<\xi \leq 1,\\
				1, &\mbox\ 1<\xi,
			\end{cases}
			\quad\
			U(\xi, t) =
			\begin{cases}
				1-\frac{1}{4}t, &\mbox\ \xi\leq 0,\\
				1- \frac{1}{4}t+\frac{(t-2)}{2}\xi, &\mbox\ 0<\xi \leq 1,\\
				\frac{1}{4}t, &\mbox\ 1<\xi,
			\end{cases}
		\]
		and
		\[
		y(\xi, t) = 
		\begin{cases}
			t-\frac{1}{8}t^2+\xi, &\mbox\ \xi\leq 0,\\
			t-\frac{1}{8}t^2+\frac{(t-2)^2}{4}\xi, &\mbox\ 0<\xi \leq 1,\\
			 \frac{1}{8}t^2+\xi &\mbox\ 1< \xi  .
		\end{cases}
		\]
		Wave breaking does not occur at $t=0$, and thus $H(\xi,t) = V(\xi, t)$ for $t<2$. See Figure \ref{fig:ex:Characteristics} for a plot of $y(\xi, t)$

		\begin{figure}
			\centering
	      \includegraphics[scale=0.3]{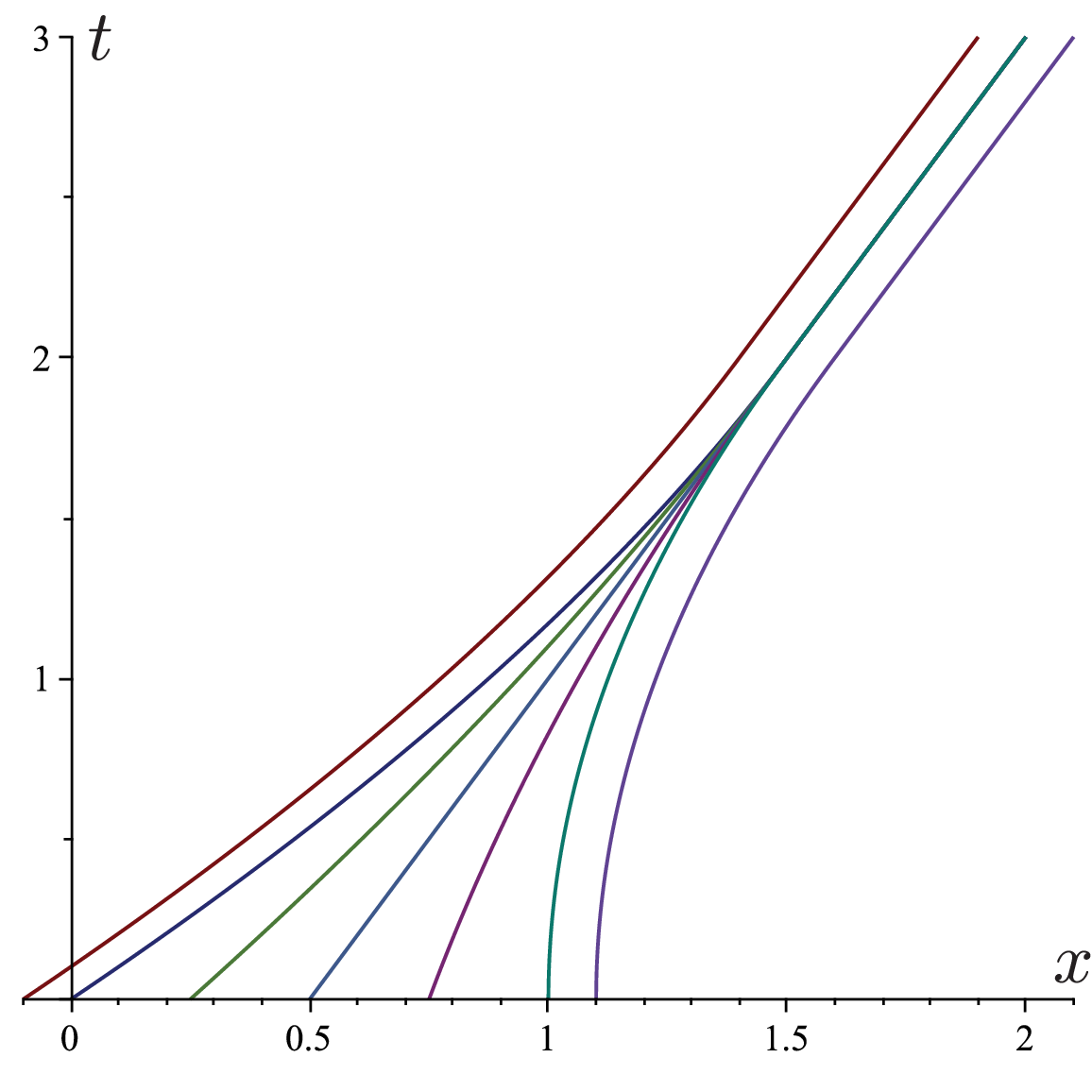}
			\caption{Characteristics $y(\xi,t)$ for Example \ref{exmp:distanceWrongEquivClasses}, for $t\in [0,3]$, in the dissipative case, i.e. $\alpha = 1$. Note how the characteristics for $\xi \in (0,1)$, meet in one point at $t=2$, and remain stuck together as all the concentrated energy is lost.}
			\label{fig:ex:Characteristics}
		\end{figure}
		
		On the other hand, we can define the initial data in Lagrangian coordinates using Definition~\ref{map:EultoLag}. This yields, using \eqref{eqn:LagSys}, for $t<2$
		\[
		\hat{V}(\xi,t) = 
		\begin{cases}
			0, &\mbox\ \xi\leq 0,\\
			\frac{1}{2}\xi, &\mbox\ 0<\xi \leq 2,\\
			1, &\mbox\ 2<\xi,
		\end{cases}
		\quad\
		\hat{U}(\xi, t) = 
		\begin{cases}
			1-\frac{1}{4}t, &\mbox\ \xi\leq 0,\\
			1-\frac{1}{4}t + \frac{(t-2)}{4}\xi, &\mbox\ 0<\xi \leq 2,\\
			\frac{1}{4}t, &\mbox\ 2<\xi,
		\end{cases}
		\]
		and
		\[
		\hat{y}(\xi, t) = 
		\begin{cases}
			t-\frac{1}{8}t^2+\xi , &\mbox\ \xi\leq 0,\\
			t-\frac{1}{8}t^2 + \frac{(t-2)^2}{8}\xi, &\mbox\ 0<\xi \leq 2,\\
			-1+\frac{1}{8}t^2+\xi , &\mbox\ 2<\xi.
		\end{cases}
		\]
		This time wave breaking occurs for all $\xi\in (0,2)$, and again $\tau(\xi)=2$ for all $\xi\in (0,2)$.
		Once again, $H(\xi, t) = V(\xi, t)$ for $t<2$. 
		
		We now wish to identify the relabelling function connecting our two solutions, which will then imply that these two solutions belong to the same equivalence class. Importantly, the distance between these two solutions is positive. Using Definition \ref{defn:equivRel} and Proposition~\ref{prop:solRel}, we see that we need to identify a homeomorphism $f$ satisfying \eqref{proprt:eqf} such that 
		\[
			(y,U,H,V)(\xi, t) =(\hat y, \hat U, \hat H,  \hat{V})(f(\xi), t).
		\]
		Since $\hat y(\xi,0)+\hat H(\xi,0)=\xi$, we see that $f\in G$ is given by
		\[
			f(\xi) = \begin{cases}
				\xi, &\mbox\ \xi \leq 0,\\
				2\xi, &\mbox\ 0 < \xi \leq 1,\\
				1+ \xi, &\mbox\ 1< \xi.
			\end{cases}
		\]
		
		For completions sake, we compute the solution using Definition~\ref{map:LagtoEul} and obtain in both cases that the solution for $t<2$ is given by
		\[
		u(x, t) =
		\begin{cases}
			1-\frac{1}{4}t, &\mbox\ x\leq t-\frac{1}{8}t^2,\\
			\frac{-4-t+4x}{2(t-2)}, &\mbox\ t-\frac{1}{8}t^2 <x \leq 1+\frac{1}{8}t^2 ,\\
			\frac{1}{4}t, &\mbox\ 1+\frac{1}{8}t^2 <x.
		\end{cases}
		\]

	\end{exmp}
	To resolve this issue, we introduce the function $J:\mF^2 \to \R$, given by
	\begin{equation}\label{def:J}
		J(X_1,X_2)=\inf_{f,g\in G} \big(d(X_1,X_2\bullet f) + d(X_1\bullet g, X_2) \big).
	\end{equation}
	This function satisfies the requirement that two elements of the same equivalence class have a distance of zero. Sadly, one cannot conclude that $J$ satisfies the triangle inequality. To resolve this issue, one constructs a metric by taking the infimum over finite sequences.
	\begin{defn}[A metric over equivalence classes in $\mF$]\label{def:dMF}
		Define the metric $d_\mF:\mF^2\to\R$ as follows
		\[
			d_\mF(X_A, X_B) \coloneqq  \inf_{\hat{\mF}(X_A, X_B) } \bigg\{ \sum_{n=1}^{N} J \big(X_n, X_{n-1} \big) \bigg\},
		\]
		where the infimum is taken over the set $\hat{\mF}(X_A, X_B) $ of finite sequences of arbitrary length $ \{ X_i\}_{i=0}^N $ in $\mF_0$, such that $X_0=\Pi X_A$ and $X_N=\Pi X_B$.
	\end{defn}

	The following lemma ensures that $d_\mF$ is indeed a metric.
	\begin{lem}\label{lem:distLowerBdd}
		Let $X_A, X_B \in \mF$ and set $(\hat{X}_A, \hat{X}_B )\coloneqq (\Pi X_A, \Pi X_B)$. We then have
		\begin{equation}\label{eqn:lem:0prop}
			\| \hat{X}_A - \hat{X}_B \|
			\leq \frac{5}{2}d_\mF(X_A, X_B) \leq 5 d(\hat{X}_A, \hat{X}_B),
		\end{equation}
		where 
		\begin{equation}\label{def:norm}
		\| {X}_A - {X}_B \| 
		\coloneqq \| y_A - y_B \|_\infty 
		+ \| {U}_A - {U}_B \|_\infty
		+ \| {H}_A - {H}_B \|_\infty
		+ \| V_A -  V_B\|_\infty.
	\end{equation}
	\end{lem}
	\begin{proof}
		The ideas of this proof follow the ones of \cite[Lemma 3.2]{MR3007728}. 
		As 
		\begin{equation*}
		d_\mF(X_A, X_B) = d_\mF(\Pi X_A, \Pi X_B), 
		\end{equation*}
		we assume for our calculations that $X_A$, $X_B \in \mF_0$.
		
		For the upper bound, consider the sequence containing just $X_A$ and $X_B$. Then
		\[
			d_\mF(X_A, X_B) \leq J(X_A, X_B)  
			= \inf_{f,g\in G} \big( d(X_A, X_B\bullet f) + d(X_A \bullet g, X_B) \big) 
			\leq 2d(X_A, X_B),
		\]
		where in the last inequality, we have chosen $f=g=\text{id}$.
		
		For the lower bound, we begin by showing that, for any $X_A, X_B \in \mF_0$,
		\[
			\|X_A - X_B\| \leq \frac{5}{2} J(X_A, X_B).
		\]
		First, for any $X\in \mF_0$, one has $X\in C^{0,1}(\R)^4$, as $Z=(y-\text{id},U,V,H) \in W^{1,\infty}(\R)^4$. Furthermore, $\|y_\xi\|_\infty$, $\|U_\xi\|_\infty$, $\|V_\xi\|_\infty$, and $\|H_\xi\|_\infty$ are all bounded from above by $1$, as $0\leq y_\xi, H_\xi\leq 1$, $0\leq V_\xi\leq H_\xi$, and $U_\xi^2 = y_\xi V_\xi \leq 1$ almost everywhere. Hence, we have
		\[
			|y(\xi_1) - y(\xi_2)| + |U(\xi_1) - U(\xi_2)| + |V(\xi_1) - V(\xi_2)| + |H(\xi_1) - H(\xi_2)| \leq 4 | \xi_1 - \xi_2 |,
		\]
		which implies, that for any $f\in G$,
		\begin{equation}\label{lem:ineq:XA-XB}
		\begin{split}
			\| X_A - X_B \| 
			&\leq \| X_A - X_A \bullet f \| + \| X_A \bullet f - X_B\|\\
			&\leq 4\| \text{id} - f \|_{\infty} + \| X_A \bullet f - X_B \|.
		\end{split}
		\end{equation}
		Then, using that $X_A \in \mF_0$, which implies $y_A + H_A = \text{id}$, and similarly for $X_B$, we get
		\[
			\| \text{id} - f \|_{\infty}  = \| y_B + H_B - (y_A + H_A) \circ f \|_\infty \leq \|X_A \bullet f - X_B \|.
		\]
		Substituting into \eqref{lem:ineq:XA-XB}, we thus end up with
		\begin{equation}\label{lem:ineq:XA-XB_2}
			\| X_A - X_B \| \leq 5 \| X_A \bullet f - X_B \|.
		\end{equation}
		Note that we have, for any $X_1$, $X_2 \in \mF$, that
		\[
			|V_1(\xi) - V_2(\xi)| 
			= \bigg| \int_{-\infty}^\xi \left( V_{1,\xi}(\eta) - V_{2,\xi}(\eta) \right)\ d\eta \bigg| 
			\leq \int_\R | V_{1,\xi}(\xi) - V_{2,\xi}(\xi) |\ d\xi 
			\leq \|G_{12}\|_1
		\]
		or equivalently
		\begin{equation}\label{lem:ineq:V1-V2}
			\|V_1 - V_2\|_\infty \leq \|G_{12}\|_1.
		\end{equation}
		Recalling \eqref{def:norm}, setting $V_1 = V_A \circ f$ and $V_2 = V_B$ in \eqref{lem:ineq:V1-V2}, and substituting into the RHS of \eqref{lem:ineq:XA-XB_2}, we get
		\begin{equation}\label{lem:ineq:XA-XB_3}
			\| X_A - X_B \| \leq 5d(X_A\bullet f, X_B).
		\end{equation}
		A similar process reveals, for any $g \in G$, that 
		\begin{equation}\label{lem:ineq:XA-XB_4}
			\| X_A - X_B \| \leq 5 \| X_A - X_B \bullet g \| \leq 5d(X_A, X_B\bullet g).
		\end{equation}
		Combining \eqref{lem:ineq:XA-XB_3} and \eqref{lem:ineq:XA-XB_4} together, and taking the infimum over all $f$, $g\in G$, we end up with
		\begin{equation}\label{lem:ineq:XA-XB_Fin}
			2\|X_A - X_B \| \leq 5J(X_A, X_B),
		\end{equation}
		as required.
		
		Consider $X_A$, $X_B \in \mF_0$. Given $\eps>0$, there exists a finite sequence $\{X_n\}_{n=0}^N$ in $\mF_0$ with $X_0=X_A$ and $X_N = X_B$, such that
		\[
			\sum_{n=1}^N J({X}_n, {X}_{n-1}) < d_\mF(X_A, X_B) + \eps.
		\]
		Using \eqref{lem:ineq:XA-XB_Fin}, we have 
		\[
			2\|X_A - X_B\| 
			\leq 2\sum_{n=1}^N \|X_n - X_{n-1}\| 
			\leq 5 \sum_{n=1}^N J({X}_n, {X}_{n-1})
			< 5 d_\mF(X_A, X_B) + 5\eps.
		\]
		Since the above inequality holds for any $\eps>0$, the claim follows. 
	\end{proof}
	
	The following lemma contains two estimates for $J$, which play en essential role when establishing the Lipschitz stablity in time for $d_\mF$. 
		
	\begin{lem}\label{thm:JHExpIneq}
		For $X_A$, $X_B\in \mF$, and $f\in G$ with $\|f_\xi\|^\hlf_\infty \leq C$ for some $C>1$, it holds that
		\[
			J(X_A\bullet f,X_B) \leq C J(X_A, X_B).
		\]
		As a consequence, for solutions $X_A(t), X_B(t) \in \mF$ of \eqref{eqn:LagSys} with initial data $X_A(0), X_B(0) \in \mF_0$, it holds that
		\[
			J(\Pi X_A(t), \Pi X_B(t)) \leq e^{\frac{1}{2}t} J(X_A(t), X_B(t)).
		\]
	\end{lem}
	
	\begin{proof} The proof follows the ideas of the one for \cite[Lemma 4.8]{MR3573580}.
		First, note for $f$, $h\in G$, and $g_A$, $g_B \in L^\infty(\R)$,
		\begin{equation}\label{lem:eqn:infNrmInv}
			\|g_A\circ f - g_B \circ h \|_\infty = \|g_A - g_B\circ h \circ f^{-1}\|_\infty.
		\end{equation}
		Importantly, due to the group properties, $w \coloneqq h\circ f^{-1}$ is in $G$. We use this relation for the $L^\infty(\R)$ terms involving $y$, $U$, and $H$ in $d$. Hence we focus on the $L^2(\R)$ and $L^1(\R)$ terms.
		
		Beginning with $L^2(\R)$ terms, for $f$, $h\in G$, we have
		\begin{align*}
			\| (y_A\circ f)_\xi - (y_B\circ h)_\xi \|_2^2
			&= \int_\R |(y_A\circ f)_\xi - (y_B\circ h)_\xi |^2(\xi)\ d\xi\\
			&=  \int_\R |y_{A,\xi}\circ f f_\xi - y_{B,\xi}\circ h h_\xi|^2(\xi)\ d\xi.
		\end{align*}
		Using the substitution  $\eta = f(\xi)$, for which $d\xi = \frac{1}{f_\xi\circ f^{-1}(\eta)}d\eta$, we have
		\[
			\| (y_A\circ f)_\xi - (y_B\circ h)_\xi \|_2^2  
			= \int_\R | y_{A,\xi} (f_\xi\circ f^{-1}) - (y_{B,\xi}\circ h \circ f^{-1}) (h_\xi \circ f^{-1})|^2(\eta) \frac{1}{f_\xi\circ f^{-1}(\eta)}d\eta.
		\]
		Using that $w=h\circ f^{-1}\in G$ has the derivative $w_\eta(\eta) = \frac{h_\xi \circ f^{-1}(\eta)}{f_\xi\circ f^{-1}(\eta)}$, we get
		\begin{equation}\label{lem:ineq:y2Nrm}
			\begin{split}
				\| (y_A\circ f)_\xi - (y_B\circ h)_\xi \|_2^2
				&= \int_\R | (y_A)_\eta - (y_{B}\circ w)_\eta|^2(\eta) f_\xi\circ f^{-1}(\eta)\ d\eta\\
				& \leq \|f_\xi \|_\infty \| (y_A)_\eta - (y_B\circ w)_\eta \|_2^2 
			\end{split}
		\end{equation}
		Similarly, one has
		\begin{equation}\label{lem:ineq:U2Nrm}
			\| (U_A\circ f)_\xi - (U_B\circ h)_\xi \|_2^2 \leq \|f_\xi \|_\infty \| (U_A)_\eta - (U_B\circ w)_\eta \|_2^2.
		\end{equation}
		For the final two norms, we need to introduce some new notation to keep everything clear. Let $X_1$ be an element of $\mF$, and using a relabelling $f \in G$ define $X_2 = X_1 \circ f$. Then we have
		\begin{align*}
			A_2 
			&= \{ \xi \in \R \mid U_{2,\xi}(\xi) \geq 0\}\\
			&= \{ \xi \in \R \mid U_{1,\xi}(f(\xi))f_\xi(\xi) \geq 0\}\\
			&= \{ \xi \in \R \mid U_{1,\xi}(f(\xi)) \geq 0\}\\
			&= \{ f^{-1}(\xi) \in \R \mid U_{1,\xi}(\xi) \geq 0\} = f^{-1}(A_1).
		\end{align*}
		Using this, we define $\Omega$ for a relabelled solution. Given $X_i, X_j \in \mF$ for some labels $i, j$, and their respective relabellings $ f, h \in G $, we define
		\[
			\Omega^{f,h}_{i,j} = (f^{-1}(A_i)\cap h^{-1}(A_j)) \cup \{ \xi\in\R \mid 0< \tau_i(f(\xi)) = \tau_j(h(\xi))<+\infty \}.
		\]
		From the same substitution as before, and using the definition of $G_{12}$,
		\begin{equation}\label{lem:eqn:G1nrm}
			\begin{split}
				&\| (V_{A}\circ f - V_{B}\ \circ h)_\xi \mathds{1}_{\Omega_{AB}^{f,h}} + ((V_{A}\circ f)_\xi \vee (V_{B}\circ h)_\xi) \mathds{1}_{\Omega_{AB}^{f,h,c}} \|_1\\
				&= \int_\R \big|(V_{A\xi}f_\xi\circ f^{-1} - (V_{B,\xi}\circ w) h_\xi \circ f^{-1}) \mathds{1}_{\Omega_{AB}^{\text{id},w}} \\
				&\hphantom{= \int_\R \big|(V}
				+ (V_{A\xi} f_\xi \circ f^{-1}) \vee \big((V_{B,\xi}\circ w) h_\xi \circ f^{-1} \big) \mathds{1}_{\Omega_{AB}^{\text{id},w,c}} \big| \frac{1}{|f_\xi\circ f^{-1}|} d\eta\\
				&= \int_\R \big|(V_{A\xi} - (V_{B}\circ w)_\xi) \mathds{1}_{\Omega_{AB}^{\text{id},w}} + V_{A\xi} \vee (V_{B}\circ w )_\xi \mathds{1}_{\Omega_{AB}^{\text{id},w,c}}\big|\ d\eta,
			\end{split}
		\end{equation}
		and similarly to before,
		\begin{equation}\label{lem:ineq:G2nrm}
			\begin{split}
			&\| (V_{A}\circ f - V_{B}\ \circ h)_\xi \mathds{1}_{\Omega_{AB}^{f,h}} + (V_{A,\xi} \vee (V_{B}\circ h)_\xi) \mathds{1}_{\Omega_{AB}^{f,h,c}} \|_2^2\\
			&\leq \|f_\xi\|_\infty \| (V_{A} - V_{B}\ \circ w)_\xi \mathds{1}_{\Omega_{AB}^{\text{id},w}} + (V_{A,\xi} \vee (V_{B}\circ w)_\xi) \mathds{1}_{\Omega_{AB}^{\text{id},w,c}} \|_2^2.
			\end{split}
		\end{equation}
		Combining \eqref{lem:eqn:infNrmInv}, \eqref{lem:ineq:y2Nrm}, \eqref{lem:ineq:U2Nrm}, \eqref{lem:eqn:G1nrm}, and \eqref{lem:ineq:G2nrm} together, we have for $f,h \in G$ and $w = h \circ f^{-1}$,
		\[
			d(X_A\bullet f, X_B \bullet h) \leq \|f_\xi\|^{\hlf}_\infty d(X_A, X_B \bullet w).
		\]
		For all these estimates, $f$ is involved in the $w$, so to ensure we can take the infimum, we assume that $\|f_\xi\|_\infty^\hlf \leq C$ for some $C>1$.
		\begin{align*}
			J(X_A\bullet f, X_B) 
			&= \inf_{f_1, f_2}(d(X_A\bullet f, X_B \bullet f_1) 
			+ d(X_A \bullet (f \circ f_2), X_B))\\
			&\leq \inf_{w_1, w_2}(C d(X_A, X_B \bullet w_1) 
			+ C d(X_A \bullet w_2, X_B))\\
			&= C\inf_{w_1, w_2}(d(X_A, X_B \bullet w_1) 
			+ d(X_A \bullet w_2, X_B)) 
			= C J(X_A, X_B),
		\end{align*}
		where we have used the fact that $w_1$ and $w_2$ above are still in the group $G$, and that given $f\in G$ for each $g\in G$, there are $h$, $l \in G$ such that $g=f\circ h=l\circ f$.
		
		Given $t$ and slightly abusing the notation, denote by $(y+H)^{-1}(\xi,t)$ the inverse of $(y+H)(\cdot, t)$. Recalling \eqref{proprt:eqf}, we have $(y+H)^{-1}(\cdot, t)\in G$. Furthermore, $(y_\xi + H_\xi)^{-1}(\xi, 0) = 1$ as $X(0)\in\mF_0$. Choose $\xi\in \R$ and drop it in the notation in the following calculation. We see that
		\begin{align*}
			\frac{d}{dt}\bigg[\frac{1}{y_\xi(t)+H_\xi(t)}\bigg] 
			= -\frac{U_\xi(t)}{(y_\xi(t)+H_\xi(t))^2} 
			&\leq \frac{1}{y_\xi(t)+H_\xi(t)} \frac{\sqrt{y_\xi(t) V_\xi(t)}}{y_\xi(t)+H_\xi(t)} \\
			&\leq \frac{1}{y_\xi(t)+H_\xi(t)} \frac{\hlf(y_\xi(t) + H_\xi(t))}{y_\xi(t)+H_\xi(t)} 
		\end{align*}
		so
		\[
			\frac{d}{dt}\bigg[\frac{1}{y_\xi(t)+H_\xi(t)}\bigg] 
			\leq \hlf\frac{1}{y_\xi(t)+H_\xi(t)},
		\]
		and hence
		\[
			\frac{1}{y_\xi(t)+H_\xi(t)} \leq e^{\hlf t}.
		\]
		Then, one has
		\[
			\big[(y+ H)^{-1}(\xi,t)\big]_\xi = \frac{1}{(y_\xi + H_\xi)(t,  (y + H)^{-1}(\xi,t))} \leq e^{\hlf t},
		\]
		
		and the result follows by using the relabeling function $f(\xi, t)=(y + H)^{-1}(\xi, t)$,
		\begin{align*}
			J\big(\Pi X_A(t), \Pi X_B(t)\big) 
			&=J\big((X_A \bullet (y_A + H_A)^{-1})(t), (X_B \bullet (y_B + H_B)^{-1})(t)\big)\\
			&\leq e^{\frac{1}{4}t} J(X_A(t), (X_B\bullet (y_B + H_B)^{-1})(t))\\
			&\leq e^{\hlf t}  J\big(X_A(t), X_B(t)\big).
		\end{align*}
	\end{proof}
	
	We can now obtain stability in Lagrangian coordinates.
	
	\begin{thm}\label{thm:LagStab}
		Let $X_A(t)$, $X_B(t)\in \mF$ be the solutions of the system \eqref{eqn:LagSys} with initial data $ X_{A}(0)$, $X_{B}(0) \in \mF_0$, respectively. Then
		\[
			d_\mF(X_A(t), X_B(t)) \leq e^{\frac{3}{2}t} d_\mF(X_{A}(0), X_{B}(0)).
		\]
	\end{thm}
	\begin{proof}
		Let $\eps>0$. There exists a finite sequence $ \{X_{n}(t)\}_{n=0}^N $ in $\mF$ of solutions to \eqref{eqn:LagSys}, whose initial data lies in $\mF_0$, and a sequence of relabelling functions $ \{f_n\}_{n=0}^{N-1}, \{g_n\}_{n=1}^N $ in $G$ such that
		\begin{align} \nonumber
		\sum_{n=1}^{N} \big(d(X_{n}(0), X_{n-1}(0)\bullet f_{n-1})  + & d(X_{n}(0)\bullet g_n,  X_{n-1}(0))\big) \\ 
		& \qquad < d_\mF(X_{A}(0), X_{B}(0)) + \eps.  \label{thm:ineq:dFeps}
		\end{align}
		From Definition~\ref{def:dMF} and Lemma \ref{thm:JHExpIneq}, it thus follows that
		\begin{align*}
			d_\mF(X_A(t), X_B(t)) 
			&\leq \sum_{n=1}^{N} J(\Pi X_n(t), \Pi X_{n-1}(t))\\
			& \leq e^{\hlf t}\sum_{n=1}^{N} J(X_n(t), X_{n-1}(t)).
		\end{align*}
		Hence, from \eqref{def:J}, Proposition \ref{prop:solRel}, and Theorem \ref{thm:est:d}, we have
		\begin{align*}
			d_\mF(X_A(t), X_B(t)) 
			&\leq e^{\hlf t} \sum_{n=1}^{N}
			\big( d(X_{n}(t), X_{n-1}(t) \bullet f_{n-1}) + d(X_{n}(t)\bullet g_n, X_{n-1}(t)) \big) \\
			&\leq e^{\frac{3}{2}t} \sum_{n=1}^{N}
			\big( d(X_{n}(0), X_{n-1}(0) \bullet f_{n-1}) + d(X_{n}(0)\bullet g_n, X_{n-1}(0)) \big)\\
			&< e^{\frac{3}{2}t} (d_\mF(X_{A}(0), X_{B}(0)) + \eps),
		\end{align*}
		where for the final inequality we have used \eqref{thm:ineq:dFeps}.
		As such a result can be constructed for $\eps$ arbitrarily small, we have 
		\[
			d_\mF(X_A(t), X_B(t)) \leq e^{\frac{3}{2}t}d_\mF(X_{A}(0), X_{B}(0)),
		\]
		as required.
	\end{proof}
	
	\section{Equivalence relation in Eulerian variables and Lipschitz stability}
	We define the metric $d_\mD:\mD^2\to\R$ on Eulerian coordinates as follows,
	\begin{equation}\label{def:dD}
		d_\mD(Y_1,Y_2) = d_\mF(L(Y_1), L(Y_2)),
	\end{equation}
	for $Y_i = (u_i, \mu_i, \nu_i)\in\mD$. An immediate consequence of Theorem \ref{thm:LagStab} is the following.
	\begin{cor}\label{cor:time}
		Let $Y_1(t),Y_2(t) \in \mD$ be the $\alpha$-dissipative solutions at time $t$ of the partial differential equation \eqref{eqn:HS}, with initial data $Y_1(0),Y_2(0) \in \mD$, then
		\[
				d_\mD(Y_1(t),Y_2(t)) \leq e^{\frac{3}{2}t} d_\mD(Y_1(0),Y_2(0)).
		\]
	\end{cor}
	As mentioned earlier, the variable $\nu$ was necessarily added to represent the past energy in the system. However, we do not supply the initial energy distribution $\nu$. The following example demonstrates that if we have two different past energy measures, our distance will be greater than zero, yet we have the same solution $(u,\mu)$ in Eulerian coordinates.
	\begin{exmp}\label{exmp:initDens}
		Consider the same $u_0$ as in Example \ref{exmp:distanceWrongEquivClasses}, but with different initial energy measures, namely
		\[
			\nu_0 = u_{0,x}^2(x) dx+\delta_2,
		\]
		and
		\[
			\mu_0 = u_{0,x}^2(x) dx + (1-\alpha)\delta_2.
		\]
		For $\alpha\not = 0$, this models the case where wave breaking takes place at $t=0$. That is, energy is initially concentrated at the point $x=2$, and an $\alpha$-part of it dissipates immediately giving rise to the difference between $\nu_0$ and $\mu_0$.
	
		Then, we have
		\[
			\nu_0((-\infty, x)) = 
			\begin{cases}
				0, &\mbox\ x \leq 0,\\
				x, &\mbox\ 0< x\leq 1,\\
				1, &\mbox\ 1 < x \leq 2,\\
				2, &\mbox\ 2 < x,
			\end{cases}
		\]
		and energy initially concentrates at $x=2$. Thus we must define our initial conditions using the mapping $L$ given by Definition \ref{map:EultoLag}. We then obtain
		\[
			y_0(\xi) = 
			\begin{cases}
				\xi, &\mbox\ \xi \leq 0,\\
				\hlf \xi, &\mbox\ 0< \xi \leq 2,\\
				-1 + \xi, &\mbox\ 2 < \xi \leq 3,\\
				2, &\mbox\ 3 < \xi \leq 4,\\
				-2 + \xi, &\mbox\ 4 < \xi,
			\end{cases}
			\quad
			U_0(\xi) =
			\begin{cases}
				1, &\mbox\ \xi\leq 0,\\
				1-\hlf \xi, &\mbox\ 0< \xi\leq 2,\\
				0, &\mbox\ 2 < \xi.
			\end{cases}
		\]
		and using $H_0(\xi) = \xi - y_0(\xi)$ and \eqref{eqn:LagSys3}, gives
		\[
		 H_0(\xi) =H(\xi,t)= 
			\begin{cases}
				0, &\mbox\ \xi \leq 0,\\
				\hlf \xi, &\mbox\ 0< \xi \leq 2,\\
				1, &\mbox\ 2 < \xi \leq 3,\\
				-2 + \xi, &\mbox\ 3 < \xi \leq 4,\\
				2, &\mbox\ 4 < \xi.
			\end{cases}
		\]
		Using formula $\eqref{def:tau}$, we find that wave breaking occurs twice. For $\xi \in (3,4)$, wave breaking occurs initially, i.e. $\tau(\xi) = 0 $ and for $\xi\in(0, 2)$ we have $\tau(\xi)=2$.
		Using \eqref{map:L} and \eqref{eqn:LagSys4}, we get, for $t<2$,
		\[
			V(\xi, t) = 
			\begin{cases}
				0, &\mbox\ \xi \leq 0,\\
				\frac{1}{2}\xi, &\mbox\ 0<\xi\leq 2,\\
				1, &\mbox\ 2<\xi\leq 3,\\
				-2 + 3\alpha + (1-\alpha)\xi, &\mbox\ 3<\xi\leq 4,\\
				2-\alpha, &\mbox\ 4<\xi.
			\end{cases}
		\]
		We then solve the Lagrangian ODE problem \eqref{eqn:LagSys} for $t\in [0,2)$, and find
		\[
			U(\xi, t) =
			\begin{cases}
				1-\frac{1}{4}(2-\alpha)t, &\mbox\ \xi\leq 0,\\
				1 - \frac{1}{4}(2-\alpha)t + \frac{1}{4}(t-2)\xi, &\mbox\ 0<\xi\leq 2,\\
				\frac{1}{4}\alpha t, &\mbox\ 2<\xi \leq 3,\\
				-\frac{1}{4} \left(6-7\alpha \right)t + \hlf(1-\alpha)t\xi, &\mbox\ 3<\xi\leq 4,\\
				\frac{1}{4}\left(2 - \alpha \right)t, &\mbox\ 4<\xi,
			\end{cases}
		\]
		and
		\[
			y(\xi,t)=
			\begin{cases}
				t -\frac{1}{8}(2-\alpha)t^2 + \xi, &\mbox\, \xi \leq 0,\\
				t - \frac{1}{8}(2-\alpha)t^2 + \frac{1}{8}(t-2)^2\xi, &\mbox\ 0 < \xi \leq 2,\\
				-1+\frac{1}{8}\alpha t^2 + \xi, &\mbox\ 2<\xi\leq 3,\\
				2 - \frac{1}{8}(6-7\alpha)t^2 + \frac{1}{4}(1-\alpha)t^2 \xi, &\mbox\ 3 \leq \xi<4,\\
				- 2 + \frac{1}{8}(2-\alpha)t^2 + \xi, &\mbox\ 4< \xi,
			\end{cases}
		\]
		see Figure \ref{fig2}.
		Note that, for any $t\in (0,2)$ and $\alpha\not=1$ the function $y(\cdot,t)$ is strictly increasing and hence invertible. 
		In particular, one has, slightly abusing the notation,
		\[
			y^{-1}(x, t) =
			\begin{cases}
				-t + \frac{1}{8}(2-\alpha)t^2 + x, &\mbox\ x \leq t-\frac{1}{8}(2-\alpha)t^2,\\
				\frac{-8t + (2-\alpha)t^2+ 8x}{(t-2)^2}, &\mbox\ t-\frac{1}{8}(2-\alpha)t^2 < x \leq 1 + \frac{1}{8}\alpha t^2,\\
				1 - \frac{1}{8}\alpha t^2 + x, &\mbox\ 1 + \frac{1}{8}\alpha t^2 < x \leq 2 + \frac{1}{8} \alpha t^2,\\
				\frac{-16 + (6 - 7\alpha)t^2 + 8x}{2(1-\alpha)t^2}, &\mbox\ 2 + \frac{1}{8} \alpha t^2 < x \leq 2 + \frac{1}{8}(2-\alpha)t^2,\\
				2 - \frac{1}{8}(2-\alpha)t^2 + x &\mbox\ 2 + \frac{1}{8}(2-\alpha)t^2 < x,
			\end{cases}
		\]
		and inserting this into $U(\xi, t)$ we obtain the solution for $t\in(0,2)$,
		\[
			u(x, t) =
			\begin{cases}
				1-\frac{1}{4}(2-\alpha)t, &\mbox\ x \leq t-\frac{1}{8}(2-\alpha)t^2,\\
				\frac{-4 -\alpha t+ 4x}{2(t-2)}, &\mbox\ t-\frac{1}{8}(2-\alpha)t^2 < x \leq 1 + \frac{1}{8}\alpha t^2,\\
				\frac{1}{4}\alpha t, &\mbox\ 1 + \frac{1}{8}\alpha t^2 < x \leq 2 + \frac{1}{8} \alpha t^2,\\
				\frac{2x-4}{t}, &\mbox\ 2 + \frac{1}{8} \alpha t^2 < x \leq 2 + \frac{1}{8}(2-\alpha)t^2,\\
				\frac{1}{4}(2-\alpha) t &\mbox\ 2 + \frac{1}{8}(2-\alpha)t^2 < x.
			\end{cases}
		\]
		The following calculations are for $\alpha \not = 1$. Using the mapping $M$, given by Definition \ref{map:LagtoEul}, we can calculate $\mu$ and $\nu$ for $t\in (0,2)$. For any Borel set $A$ of $\R$, we get
		\begin{align*}
			\mu(A,t) 
			&= \int_{y^{-1}(A, t)} V_\xi(\xi,t)\ d\xi  \\
			&= \int_{y^{-1}(A, t)} \hlf\mathds{1}_{(0,2]}(\xi)\ d\xi+  \int_{y^{-1}(A, t)} (1-\alpha)\mathds{1}_{(3,4]}(\xi)\ d\xi\\
			&= \int_{y^{-1}(A \cap (t-\frac{1}{8}(2-\alpha)t^2, 1 + \frac{1}{8}\alpha t^2], t)} \hlf\ d\xi \\
			& \qquad + \int_{y^{-1}(A \cap (2+\frac{1}{8}\alpha t^2, 2 + \frac{1}{8}(2-\alpha) t^2], t)} (1-\alpha) d\xi\\
			&= \hlf \int_A \mathds{1}_{(t-\frac{1}{8}(2-\alpha)t^2, 1 + \frac{1}{8}\alpha t^2]} (y^{-1}(x,t))_x\ dx \\
			 & \qquad +  (1-\alpha) \int_A \mathds{1}_{ (2+\frac{1}{8}\alpha t^2, 2 + \frac{1}{8}(2-\alpha) t^2]} (y^{-1}(x,t))_x\ dx\\
			&= \frac{4}{(t-2)^2}\int_A \mathds{1}_{(t-\frac{1}{8}(2-\alpha)t^2, 1 + \frac{1}{8}\alpha t^2]}(x)\ dx\\
			& \qquad + \frac{4}{t^2}\int_A \mathds{1}_{(2+\frac{1}{8}\alpha t^2, 2 + \frac{1}{8}(2-\alpha) t^2]}(x)\ dx\\
			&= \int_A u_x^2(x,t)\ dx,
		\end{align*}
		and for $\nu$, we find
		\begin{align*}
			\nu(A,t) 
			= \int_{y^{-1}(A, t)} H_\xi(\xi,t)\ d\xi  
			&= \int_{y^{-1}(A, t)} \left(\hlf\mathds{1}_{(0,2]}(\xi) + \mathds{1}_{(3, 4]}(\xi)\right)\ d\xi\\
			&= \int_A u_x^2(x,t)\ dx + \alpha \int_{y^{-1}(A \cap (2+\frac{1}{8}\alpha t^2, 2 + \frac{1}{8}(2-\alpha) t^2], t)} d\xi\\
			&= \mu(A, t) + 4\frac{\alpha}{(1-\alpha)t^2}\int_A \mathds{1}_{(2+\frac{1}{8}\alpha t^2, 2 + \frac{1}{8}(2-\alpha) t^2]}(x)\ dx.
		 \end{align*}
		 Similar calculations yield for $\alpha=1$ and any Borel set $A$ of $\R$,
		 \begin{align*}
		 \mu(A,t)&=\int_A u_x^2(x,t)\ dx,\\
		 \nu(A,t)&= \mu(A, t) + \delta_{\{2+\frac{t^2}{8}\}}(A).
		 \end{align*}
		 
		 We can now compare this example with $\alpha=1$ to Example \ref{exmp:distanceWrongEquivClasses}. Both choices of $\nu_0$ lead to the same solution $(u,\mu)$ in Eulerian coordinates. So, for the given initial data $(u_0, \mu_0)$, there is an equivalence class consisting of triplets $(u_0,\mu_0,\nu_0)$ leading to the same solution $(u,\mu)$. However, different choices of $\nu$ lead to quadruples in Lagrangian coordinates that cannot be identified using relabeling and hence their distance with respect to $d_\mD$, cf. \eqref{def:dD}, will be greater than zero.
		\begin{figure}[!tbp]
		  \centering
		  \subfloat[$\alpha = 1$]{\includegraphics[width=0.4\textwidth]{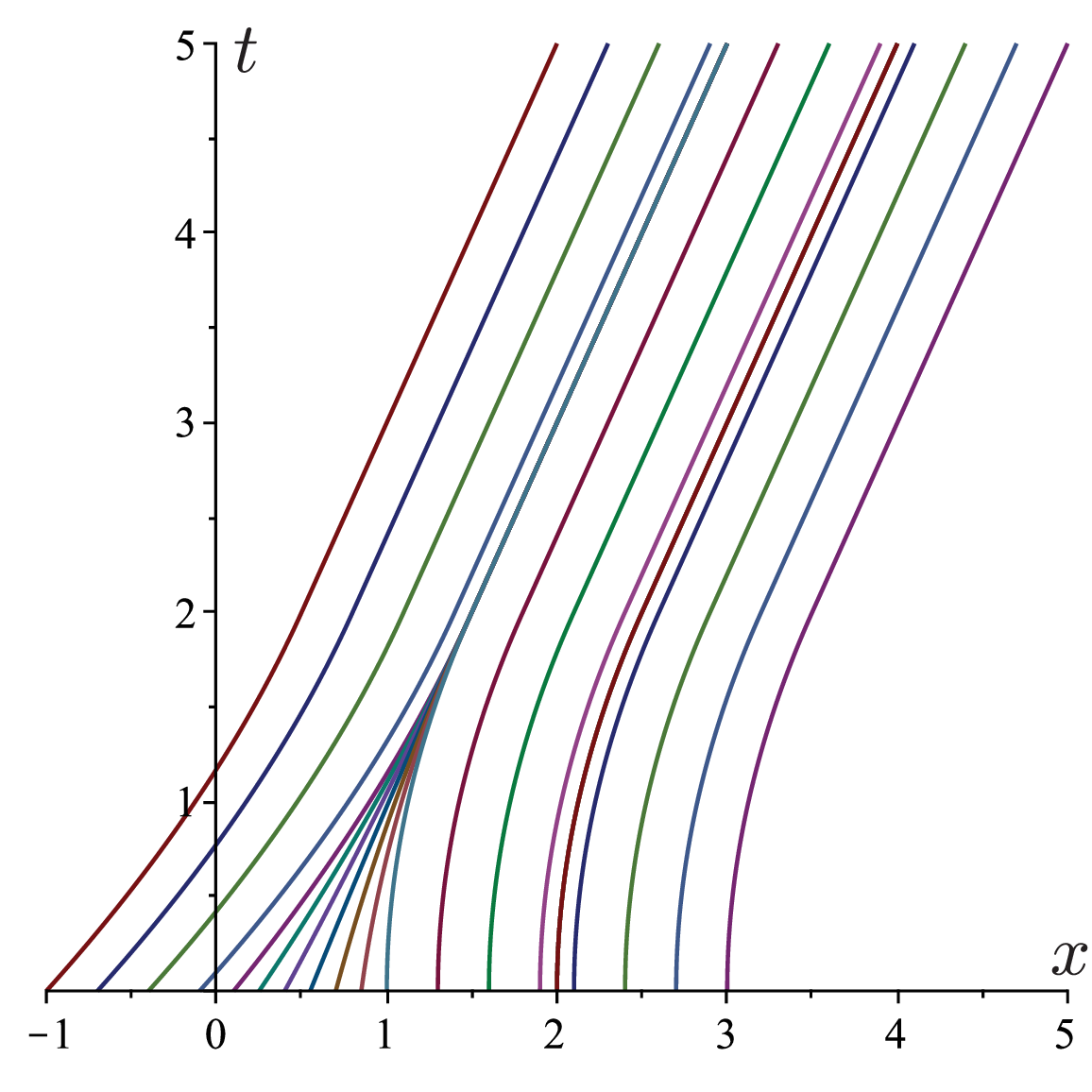}\label{fig:a1}}
		  \hfill
		  \subfloat[$\alpha=0.5$]{\includegraphics[width=0.4\textwidth]{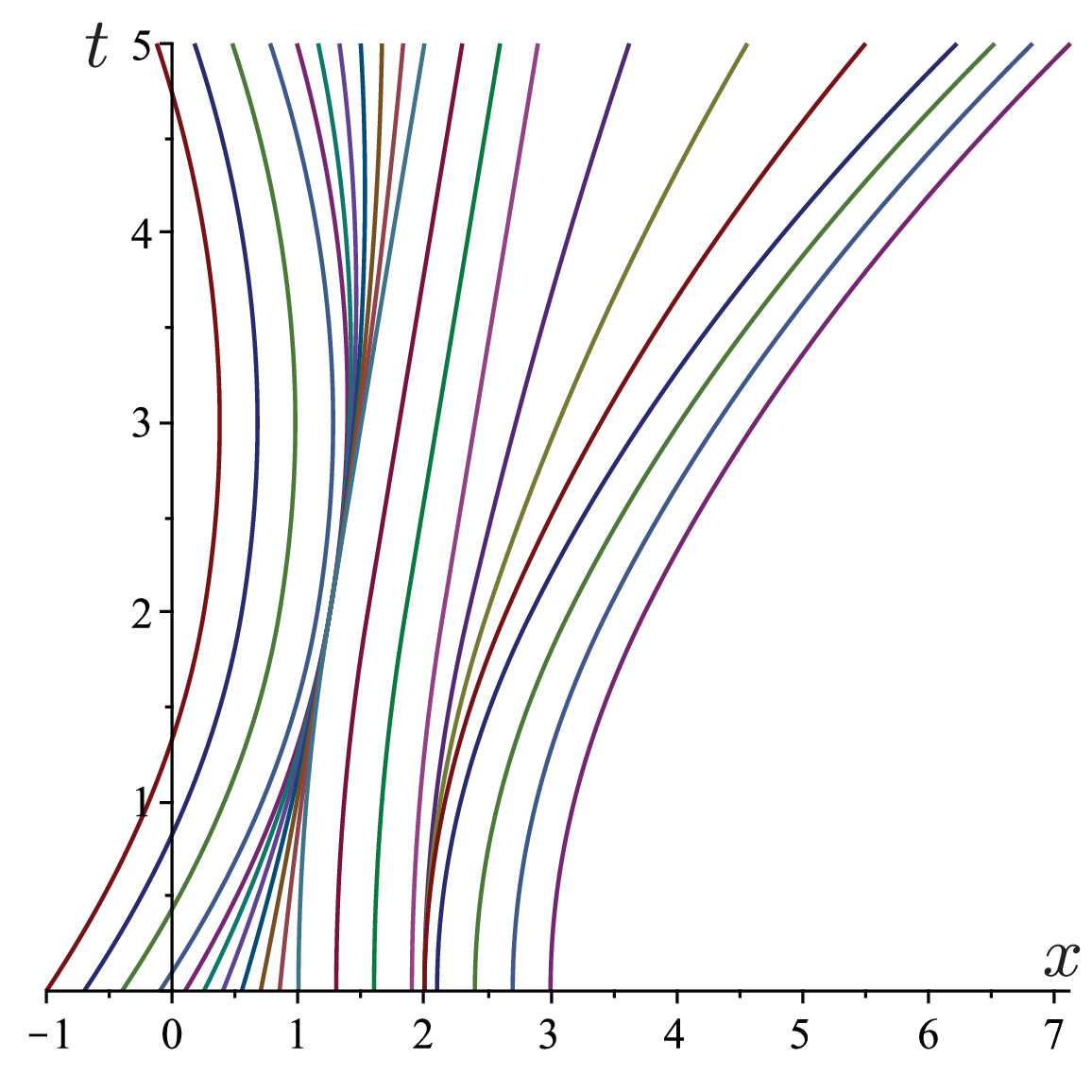}\label{fig:a0.5}}
		  \caption{Plots of the characteristics for the initial data in Example \ref{exmp:initDens}. Note the initial density causes characteristics to grow from a single point in the $\alpha = 0.5$ case, while in the $\alpha = 1$ case the loss of energy causes them to stick together.}
		  \label{fig2}
		\end{figure}
	\end{exmp}
	We do not know $\nu$, hence when going backwards in time our metric in Eulerian coordinates can only be defined using $u$ and $\mu$. We define the metric in a similar way to how we defined our $J$ in the previous section.
	We first define the set $\mD_{0,M}$, which is our original set $\mD$ without the $\nu$, with an additional assumption that our energy measure is bounded. This will be necessary to ensure that our construction satisfies the definition of a metric. Let
	\begin{equation}\label{eqn:def:D0M}
		\mD_{0,M} \coloneqq \big\{ (u, \mu) \in E_2 \times \mathcal{M}^+(\R) \mid  \mu_{ac} = u_x^2\ dx, \mu(\R) \leq M, \text{ and  }\mu = u_x^2\ dx \text{ if } \alpha=1 \big\}.
	\end{equation}
	Then, for $\hY=(u,\mu) \in \mD_{0,M}$, define the set $\mathcal{V}(\hY)$ to be the set of all $\nu\in \mathcal{M}^+(\R)$ satisfying
	\begin{itemize}
		\item $ \mu \leq \nu \in \mathcal{M}^+(\R) $,
		\item $ \mu\big((-\infty, x)\big) - \chi_+(x)\mu(\R) \in L^2(\R)$
		\item If $\alpha=1$, $ \nu_{ac} = \mu = u_x^2\ dx $,
		\item If $0\leq \alpha < 1$, $ \frac{d\mu}{d\nu}(x) \in \{1, 1-\alpha\}$, and $\frac{d\mu}{d\nu} = 1$ if $u_x(x)<0$.
	\end{itemize}
	Consider $(u,\mu)\in \mD_{0,M}$. We note the following inequality,	
	\begin{equation}\label{ineq:ux2M}
		\int_\R u_x^2(x)\ dx \leq \mu(\R) \leq M.
	\end{equation}
	
	Define the mapping $J_\mD:\mD^2_{0,M} \to \R$ as
	\begin{equation}\label{eqn:JmD}
		J_\mD(\hY_1, \hY_2) = \inf_{(\nu_1, \nu_2) \in \mathcal{V}(\hY_1)\times\mathcal{V}(\hY_2)} d_\mD((u_1, \mu_1, \nu_1), (u_2, \mu_2, \nu_2)).
	\end{equation}
	We encounter a similar problem as to our metric on the previous set of equivalence classes in $\mF$. We cannot conclude that the triangle inequality is satisfied for this distance. 
	
	Following a similar construction as before, we define the metric $d_M:\mD_{0,M}^2 \to \R$ by
	\begin{equation}\label{def:dM}
		d_M(\hY_A, \hY_B) \coloneqq \inf_{\hat{\mD}(Y_A, Y_B)} \sum_{n=1}^N J_\mD(\hY_n, \hY_{n-1}),
	\end{equation}
	where the infimum is taken over ${\hat{\mD}(\hY_1, \hY_2)}$, the set of all finite sequences $\{\hY_i\}_{i=1}^N$ in $\mD_{0,M}$ satisfying $\hY_0 = \hY_A$ and $\hY_N = \hY_B$. The following result ensures this is a metric.
	\begin{lem}
		The function $d_M:\mD_{0,M}^2 \to \R$ given by \eqref{def:dM} defines a metric on $\mD_{0,M}$.
	\end{lem}
	\begin{proof}
		Symmetry is immediate, as the distance $d_M$, if you dig deep enough, is constructed of metrics.\\
		The triangle inequality is more challenging. Let $\hY_A$, $\hY_B$, $\hY_C\in \mD_{0, M}$. Choose $\eps >0$. Select two sequences
		\begin{itemize}
			\item $ \{\hY_i\}_{i=0}^{N_1}$ in $ \hat{\mD}(\hY_A, \hY_B) $, and 
			\item $ \{\hY_i\}_{i=N_1}^{N_2}$ in $ \hat{\mD}(\hY_B, \hY_C) $,
		\end{itemize}
		where $N_1, N_2\in\N$ and $N_1<N_2$, such that 
		\begin{itemize}
			\item $ \sum_{n=1}^{N_1}  J_\mD(\hY_n, \hY_{n-1}) \leq d_M(\hY_A, \hY_B) + \eps$, and
			\item $ \sum_{n=N_1+1}^{N_2}  J_\mD(\hY_n, \hY_{n-1}) \leq d_M(\hY_B, \hY_C) + \eps$.
		\end{itemize}
		Then
		\begin{align*}
			d_M(\hY_A, \hY_C) 
			\leq \sum_{n=1}^{N_2}  J_\mD(\hY_n, \hY_{n-1}) 
			&= \sum_{n=1}^{N_1}  J_\mD(\hY_n, \hY_{n-1}) 
			+ \sum_{n=N_1+1}^{N_2}  J_\mD(\hY_n, \hY_{n-1}) \\
			&\leq d_M(\hY_A, \hY_B) + d_M(\hY_B, \hY_C) + 2\eps.
		\end{align*}
		As one can make a similar construction for any $\eps >0$, the inequality involving the RHS and LHS is satisfied for any $\eps>0$, and hence
		\[
			d_M(\hY_A, \hY_C) \leq d_M(\hY_A, \hY_B) + d_M(\hY_B, \hY_C).
		\]
		It remains to show the zero condition, that is
		\[
			d_M(\hY_A, \hY_B) = 0 \quad \text{ if and only if } \quad \hY_A = \hY_B.
		\]
		First, set $\hY = \hY_A = \hY_B$, and let $\nu \in \mathcal{V}(\hY)$, we have 
		\[
			0 \leq d_M(\hY, \hY) \leq d_\mD((\hat{u},\hat{\mu},\nu), (\hat{u}, \hat{\mu}, \nu)) = 0.
		\]
		Thus we obtain the backward implication for this statement. The forward implication is more challenging.
		
		Suppose $ d_M(\hY_A, \hY_B) = 0 $. Let $\eps>0$, and select a sequence $\{Y_n\}_{n=0}^N$ in $\mD$ with $\mu_n(\R) \leq M$ for all $n$, $(u_0, \mu_0) = (u_A, \mu_A)$, and $(u_N, \mu_N) = (u_B, \mu_B)$, such that
		\[
			\sum_{n=1}^N d_\mD(Y_n, Y_{n-1}) 
			< d_M(\hY_A, \hY_B) + \frac{2}{5}\eps = \frac{2}{5}\eps.
		\]
		Such a sequence exists because of the definition of the infimum.
		
		Setting $X_n = L(Y_n)$, and using Lemma \ref{lem:distLowerBdd} together with \eqref{def:dD}, we have
		\begin{equation}\label{est:Xnorm}
			\sum_{n=1}^N \|X_n-X_{n-1}\| \leq\frac{5}{2} \sum_{n=1}^N d_\mD(Y_n, Y_{n-1}) < \eps.
		\end{equation}
		Immediately from the definition of the norm $\|\cdot\|$, given by \eqref{def:norm}, we have that
		\begin{equation}\label{eqn:UYeps}
			\sum_{n=1}^N \|y_n-y_{n-1}\|_\infty < \eps \quad \text{ and }\quad  \sum_{n=1}^N \|U_n-U_{n-1}\|_\infty < \eps.
		\end{equation}
		Let $X_A=X_0=L(Y_0)$ and $X_B=X_N=L(Y_N)$. Note that $y_A$ and $y_B$ are continuous and increasing, by Definition \eqref{map:EultoLag}. Thus for any $x\in\R$, there are $\xi_A$ and $\xi_B$ such that $y_A(\xi_A)=x=y_B(\xi_B)$. Substituting this into the difference of the $u$'s, we get
		\begin{align*}
			|u_A(x)-u_B(x)| 
			&= |u_A(y_A(\xi_A)) - u_B(y_A(\xi_A))|\\
			&\leq |u_A(y_A(\xi_A)) - u_B(y_B(\xi_A))| + |u_B(y_B(\xi_A)) - u_B(y_A(\xi_A))|\\
			&= |U_A(\xi_A) - U_B(\xi_A)| + \left| \int_{y_A(\xi_A)}^{y_B(\xi_A)} u_{B,x}(x)\ dx \right|\\
			&\leq \|U_A-U_B\|_\infty + \sqrt{|y_A(\xi_A)-y_B(\xi_A)|} \sqrt{\left| \int_{y_A(\xi_A)}^{y_B(\xi_A)} u_{B,x}^2(x)\ dx\right |}\\
			&\leq \sum_{n=1}^N \|U_n-U_{n-1}\|_\infty + \sqrt{\sum_{n=1}^N \|y_n-y_{n-1}\|_\infty}\sqrt{M}\\
			&< \eps + \sqrt{\eps M},
		\end{align*}
		where we have used the Cauchy Schwartz inequality to split our integral, and \eqref{ineq:ux2M}. 
		As this is satisfied for any $\eps > 0$, one has $u_A = u_B$.

		We now show $\mu_A=\mu_B$. 
		From \cite[Section 7.3]{MR1681462}, we need only to show that
		\begin{equation}\label{eqn:fmue}
			\int_\R f(x)\ d\mu_A(x) 
			= \int_\R f(x)\ d\mu_B(x), 
			\quad\ \text{ for all } f\in C_0(\R),
		\end{equation}
		where  $C_0(\R)$ denotes the set of all continuous functions whom vanish at $\pm \infty$. Using that $C_c^\infty(\R)$ is a dense subset of $C_0(\R)$, it suffices to show \eqref{eqn:fmue} for any $f\in C_c^\infty(\R)$.
		
		Let $f \in C_c^\infty(\R)$, then
		\begin{align*}
			\int_\R f(x) (d\mu_A - d\mu_B)(x) 
			&= \int_\R[ (f\circ y_A)(\xi)V_{A,\xi}(\xi) - (f\circ y_B)(\xi)V_{B,\xi}(\xi)]\ d\xi\\
			&= \int_\R (f\circ y_A)(\xi)(V_{A,\xi}(\xi)-V_{B,\xi}(\xi))\ d\xi\\
			&\quad\ +\int_\R [(f\circ y_A)(\xi) - (f\circ y_B)(\xi)]V_{B,\xi}(\xi)\ d\xi
		\end{align*}
		We show these two integrals equal zero.
		
		For the first of these two integrals use integration by parts,
		\[
			\int_\R (f\circ y_A)(\xi)(V_{A,\xi}(\xi)-V_{B,\xi}(\xi))\ d\xi = 
			- \int_\R y_{A,\xi}(\xi) (f'\circ y_A)(\xi)(V_{A}(\xi)-V_{B}(\xi))\ d\xi.
		\]
		Using that $0\leq y_{A,\xi}\leq 1$, we have that
		\[
			 \int_\R |y_{A,\xi}(\xi) (f'\circ y_A)(\xi)(V_{A}(\xi)-V_{B}(\xi))|\ d\xi
			\leq \|f'\|_1 \|V_A - V_B\|_\infty \leq \|f'\|_1 \eps,
		\]
		where we have used that \eqref{est:Xnorm} implies
		\[
			\|V_A - V_B\|_\infty \leq \sum_{n=1}^N \|X_n - X_{n-1}\| < \eps.
		\]
		
		For the second integral, we use
		\begin{align*}
			\int_\R |(f\circ y_A)(\xi) - (f\circ y_B)(\xi)|V_{B,\xi}(\xi)\ d\xi
			&\leq \|(f\circ y_A)(\xi) - (f\circ y_B)(\xi)\|_\infty \|V_{B,\xi}\|_1.
		\end{align*}
		We have that $\|V_{B,\xi}\|_1 \leq M$. Also,
		\[
			|(f\circ y_A)(\xi) - (f\circ y_B)(\xi)| 
			\leq \left| \int_{y_A(\xi)}^{y_B(\xi)} f'(\eta)\ d\eta \right| 
			\leq \|y_B - y_A\|_\infty \|f'\|_\infty\\
			< \eps \|f'\|_\infty,
		\]
		and thus
		\[
			\|(f\circ y_A)(\xi) - (f\circ y_B)(\xi)\|_\infty \|V_{B,\xi}\|_1 < \eps \|f'\|_\infty M.
		\]
		Once again, this is true for any $\eps>0$, and hence the integrals are zero, concluding the proof.
	\end{proof}
	
	From this, we can conclude with our final Lipschitz stability result.
	\begin{thm}
		Let $\hY_A(t)=(u_A, \mu_A)(t)$ and $\hY_B(t)=(u_B, \mu_B)(t)$ be $\alpha$-dissipative solutions at time $t$ to the problem
		\begin{equation}\label{eqn:HSFinThm}
			u_t(x,t) + uu_x(x,t) = \frac{1}{4}\Bigg(\int_{-\infty}^x u_x^2(y,t)\ dy - \int_x^{+\infty}u_x^2(y,t)\ dy\Bigg),
		\end{equation}
		with initial data $\hY_{A}(0), \hY_{B}(0) \in \mD_{0,M}$ respectively. Then
		\[
			d_M(\hY_A(t),\hY_B(t)) \leq e^{\frac{3}{2} t}d_M(\hY_{A}(0),\hY_{B}(0)).
		\]
	\end{thm}
	\begin{proof}
		Let $\eps > 0$, and choose a finite sequence $ \{Y_i(t)\}_{i=0}^N $ of $\alpha$-dissipative solutions to the partial differential equation \eqref{eqn:HSFinThm} in $\mD$, with initial data$ \{Y_i(0)\}_{i=0}^N $ in $\mD$ satisfying $(u_0, \mu_0)(0)=(u_A,\mu_A)(0), (u_N, \mu_N)(0)=(u_B,\mu_B)(0)$, $\mu_i(\R)\leq M$ for all $i=1,\dots ,N$, and such that
		\[
			\sum_{n=1}^N d_\mD(Y_n(0),Y_{n-1}(0)) < d_M(\hY_{A}(0), \hY_{B}(0)) + \eps.
		\]
		Then, we have using Corollary~\ref{cor:time}
		\begin{align*}
			d_M(\hY_A(t),\hY_B(t)) 
			&\leq \sum_{n=1}^N d_\mD\Big(Y_n(t),Y_{n-1}(t)\Big)\\
			&\leq e^{\frac{3}{2}t}\sum_{n=1}^N d_\mD\Big(Y_n(0),Y_{n-1}(0)\Big)\\
			&< e^{\frac{3}{2}t}(d_M(\hY_{A}(0), \hY_{B}(0)) + \eps).
		\end{align*}
		As one can construct such a relation for any $\eps>0$, we obtain the required result. 
	\end{proof}

	\appendix
	\section{Examples}
	\begin{exmp}\label{exmp:adiss}
		We compute an $\alpha$-dissipative example with $\alpha = \frac{1}{3}$. Given
		\[
			u_0(x) = 
			\begin{cases}
				1 ,&\mbox\ x\leq -2 ,\\
				-1-x ,&\mbox\ -2 < x\leq -1 ,\\
				0 ,&\mbox\ -1 < x\leq 1 ,\\
				1-x ,&\mbox\ 1 < x\leq 2 ,\\
				-1 ,&\mbox\ 2<x,
			\end{cases}
			\quad\
			\mu_0=\nu_0=u_{0,x}^2(x)\ dx,
		\]
		so that
		\[
			\mu_0((-\infty,x))=\nu_0((-\infty,x))=
			\begin{cases}
				0, &\mbox\ x\leq -2,\\
				x+2, &\mbox\ -2 < x\leq -1,\\
				1, &\mbox\ -1 < x\leq 1,\\
				x, &\mbox\ 1 < x\leq 2,\\
				2, &\mbox\ 2 < x,
			\end{cases}
		\]
		then the transformation $L$, given by Definition~\ref{map:EultoLag}, yields
		\[
			y_0(\xi) \coloneqq
			\begin{cases}
				\xi ,&\mbox\ \xi\leq -2 ,\\
				-1+\frac{1}{2}\xi ,&\mbox\ -2 < \xi\leq 0 ,\\
				-1+\xi ,&\mbox\ 0 < \xi\leq 2 ,\\
				\frac{1}{2}\xi ,&\mbox\ 2 < \xi\leq 4 ,\\
				-2+\xi ,&\mbox\ 4 < \xi,
			\end{cases}
			\quad\
			U_0(\xi) =
			\begin{cases}
				1 ,&\mbox\ \xi\leq -2 ,\\
				-\frac{1}{2}\xi ,&\mbox\ -2 < \xi\leq 0 ,\\
				0 ,&\mbox\ 0 < \xi\leq 2 ,\\
				1-\frac{1}{2}\xi ,&\mbox\ 2 < \xi\leq 4 ,\\
				-1 ,&\mbox\ 4< \xi,
			\end{cases}
		\]
		and
		\[
			V_0(\xi)=H_0(\xi) =
			\begin{cases}
				0 ,&\mbox\ \xi\leq -2 ,\\
				1+\frac{1}{2}\xi ,&\mbox\ -2 < \xi\leq 0 ,\\
				1 ,&\mbox\ 0 < \xi\leq 2 ,\\
				\frac{1}{2}\xi ,&\mbox\ 2 < \xi\leq 4 ,\\
				2 ,&\mbox\ 4 < \xi.
			\end{cases}
		\]
		Next, we determine for which points $\xi\in \mathbb{R}$  wave breaking will occur and when. Using \eqref{def:tau}, we have
		\[
			\tau(\xi) =
			\begin{cases}
				2 ,&\mbox\ \xi\in (-2,0)\cup (2,4) ,\\
				\infty ,&\mbox\ \text{otherwise}.
			\end{cases}
		\]
		Computing the solution using \eqref{eqn:LagSys}, one obtains
		\[
			y(\xi, t) =
				\begin{cases}
					\begin{cases}
						 t - \frac{1}{4}t^2+\xi , &\mbox\ \xi \leq -2,\\
						-1 + \frac{(t-2)^2}{8}\xi, &\mbox\ -2 <  \xi \leq 0,\\
					       - 1+\xi, &\mbox\ 0 < \xi \leq 2,\\
						t-\frac14 t^2+\frac{(t-2)^2}{8}\xi, &\mbox\ 2 < \xi \leq 4,\\
						-2 -t + \frac{1}{4}t^2+\xi, &\mbox\ 4<\xi,\\
					\end{cases} 
					&\mbox\ 0\leq t<2,\\
					\begin{cases}
						 \frac{1}{3}+\frac{2}{3}t-\frac16 t^2+\xi, &\mbox\ \xi \leq -2,\\
						-1 + \frac{(t-2)^2}{12}\xi, &\mbox\ -2 < \xi \leq 0,\\
						- 1+\xi &\mbox\ 0 < \xi \leq 2,\\
						\frac13 +\frac23 t -\frac16 t^2+\frac{(t-2)^2}{12}\xi, &\mbox\ 2 < \xi \leq 4,\\
						- \frac{7}{3} - \frac{2}{3}t+\frac16 t^2+\xi, &\mbox\ 4< \xi,\\
					\end{cases}
					&\mbox\ 2\leq t,
				\end{cases}
		\]
		\[
			U(\xi, t) =
			\begin{cases}
				\begin{cases}
					1-\frac{1}{2}t, &\mbox\ \xi \leq -2,\\
					\frac{(t-2)}{4}\xi, &\mbox\ -2 < \xi \leq 0,\\
					0, &\mbox\ 0 < \xi \leq 2,\\
					1-\frac12 t+\frac{(t-2)}{4}\xi, &\mbox\ 2 < \xi \leq 4,\\
					-1+\frac{1}{2}t, &\mbox\ 4 < \xi,\\
				\end{cases} 
				&\mbox\ 0\leq t<2,\\
				\begin{cases}
					\frac{2}{3}-\frac{1}{3}t, &\mbox\ \xi \leq -2,\\
					\frac{(t-2)}{6}\xi, &\mbox\ -2 < \xi \leq 0,\\
					0, &\mbox\ 0 < \xi \leq 2,\\
					\frac23 -\frac13 t+\frac{(t-2)}{6}\xi, &\mbox\ 2 < \xi \leq 4,\\
					-\frac{2}{3}+ \frac{1}{3}t, &\mbox\ 4 < \xi,\\
				\end{cases}
				&\mbox\ 2\leq t,
			\end{cases}
		\]

		\[ H(\xi,t)=H_0(\xi), \quad 0\leq t,\]
		and 
		\[
			V(\xi, t) =
			\begin{cases}
				H(\xi), &\mbox\ 0\leq t<2,\\
				\begin{cases}
					0, &\mbox\ \xi \leq -2,\\
					\frac{2}{3}+\frac{1}{3}\xi, &\mbox\ -2 < \xi \leq 0,\\
					\frac{2}{3}, &\mbox\ 0 < \xi \leq 2,\\
					\frac{1}{3}\xi, &\mbox\ 2 < \xi \leq 4,\\
					\frac{4}{3}, &\mbox\ 4 < \xi,\\
				\end{cases}
				&\mbox\ 2\leq t.
			\end{cases}
		\]
		
 Using Definition~\ref{map:LagtoEul}, we can finally compute the solution $(u,\mu,\nu)$, which is given by
 		\[
			u(x, t) = 
			\begin{cases}
				\begin{cases}
					1-\frac{1}{2}t, &\mbox\ x\leq -2+t-\frac{1}{4}t^2,\\
					\frac{2+2x}{t-2}, &\mbox\ -2+t-\frac{1}{4}t^2 <  x\leq -1,\\
					0, &\mbox\ -1 < x\leq 1,\\
					\frac{-2+2x}{t-2}, &\mbox\ 1 < x\leq 2-t+\frac{1}{4}t^2,\\
					-1+\frac{1}{2}t, &\mbox\ 2-t+\frac{1}{4}t^2 < x,
				\end{cases}
				&\mbox\ t<2,\\
				0, &\mbox\ t=2\\
				\begin{cases}
					\frac{2}{3}-\frac{1}{3}t, &\mbox\ x\leq -\frac{5}{3}+\frac23 t-\frac{1}{6}t^2,\\
					\frac{2+2x}{t-2}, &\mbox\ -\frac{5}{3}+\frac23 t-\frac{1}{6}t^2 < x\leq -1,\\
					0, &\mbox\ -1 < x\leq 1,\\
					\frac{-2+2x}{t-2}, &\mbox\ 1 < x\leq \frac{5}{3}-\frac{2}{3}t+\frac{1}{6}t^2,\\
					-\frac{2}{3}+\frac13 t, &\mbox\ \frac{5}{3}-\frac{2}{3}t+\frac{1}{6}t^2 < x,
				\end{cases}
				&\mbox\  2<t,
			\end{cases}
		\]
		\[
			\mu(t,(-\infty, x)) = 
			\begin{cases}
				\begin{cases}
					0, &\mbox\ x\leq -2+t-\frac{1}{4}t^2,\\
					1+\frac{4+4x}{(t-2)^2}, &\mbox\ -2+t-\frac{1}{4}t^2 < x\leq -1,\\
					1, &\mbox\ -1 < x\leq 1,\\
					1+\frac{-4+4x}{(t-2)^2}, &\mbox\ 1 < x\leq 2-t+\frac{1}{4}t^2,\\
					2, &\mbox\ 2-t+ \frac{1}{4}t^2 < x,
				\end{cases}
				&\mbox\ t<2,\\
				\begin{cases}
						0, &\mbox\ x\leq -1,\\
						\frac{2}{3}, &\mbox\ -1 < x\leq 1,\\
						\frac{4}{3}, &\mbox\ 1<x,
				\end{cases}
				&\mbox\ t=2\\
				\begin{cases}
					0, &\mbox\ x\leq -\frac{5}{3}+\frac23 t-\frac{1}{6}t^2,\\
					\frac{2}{3}+\frac{4+4x}{(t-2)^2}, &\mbox\ -\frac{5}{3}+\frac23 t-\frac{1}{6}t^2 <  x\leq -1,\\
					\frac{2}{3}, &\mbox\ -1 < x\leq 1,\\
					\frac{2}{3}+\frac{-4+4x}{(t-2)^2}, &\mbox\ 1 < x\leq \frac{5}{3}-\frac{2}{3}t+\frac{1}{6}t^2,\\
					\frac{4}{3}, &\mbox\ \frac{5}{3}-\frac{2}{3}t+\frac{1}{6}t^2 < x,
				\end{cases}
				&\mbox\  2<t,
			\end{cases}
		\]
		and 
		\[
			\nu(t,(-\infty, x)) = 
			\begin{cases}
				\begin{cases}
					0, &\mbox\ x\leq -2+t-\frac{1}{4}t^2,\\
					1+\frac{4+4x}{(t-2)^2}, &\mbox\ -2+t-\frac{1}{4}t^2 < x\leq -1,\\
					1, &\mbox\  -1 < x\leq 1,\\
					1+\frac{-4+4x}{(t-2)^2}, &\mbox\ 1 < x\leq 2-t+\frac{1}{4}t^2,\\
					2, &\mbox\ 2-t+ \frac{1}{4}t^2 < x,
				\end{cases}
				&\mbox\ t<2,\\
				\begin{cases}
						0, &\mbox\ x\leq -1,\\
						1, &\mbox\ -1 < x\leq 1,\\
						2, &\mbox\ 1<x,
				\end{cases}
				&\mbox\ t=2\\
				\begin{cases}
					0, &\mbox\ x\leq -\frac{5}{3}+\frac23 t-\frac{1}{6}t^2,\\
					1+\frac{6+6x}{(t-2)^2}, &\mbox\ -\frac{5}{3}+\frac23 t-\frac{1}{6}t^2 <  x\leq -1,\\
					1, &\mbox\ -1 < x\leq 1,\\
					1+\frac{-6+6x}{(t-2)^2}, &\mbox\ 1 < x\leq \frac{5}{3}-\frac{2}{3}t+\frac{1}{6}t^2,\\
					2, &\mbox\ \frac{5}{3}-\frac{2}{3}t+\frac{1}{6}t^2 < x,
				\end{cases}
				&\mbox\ 2<t.
			\end{cases}
		\]
		Notice that $\nu$ carries the initial energy forward in time, while $\mu$ is the actual energy in the system at the current time. Thus the difference in the two is the lost energy.	
		\end{exmp}

\end{document}